\documentclass[reqno, 11pt, openany]{amsart}
\pdfoutput=1 
\usepackage[usenames, dvipsnames, table]{xcolor}
\usepackage{algorithm,algorithmic}

\usepackage{placeins}
\usepackage{amsmath,amssymb,amsthm,mathtools}
\usepackage{tikz}
\usepackage[T1]{fontenc}
\usepackage[utf8]{inputenc}
\usepackage{enumerate}

\newtheorem*{theorem*}{}
\usepackage[letterpaper, margin=1in]{geometry}
\usepackage[noBBpl]{mathpazo} 
\usepackage{tablefootnote} 
\usepackage{xspace} 
\usepackage{listings}
\usepackage{hyperref}
\usepackage{lipsum}
\usepackage[LSBC5,T1]{fontenc}

\usepackage{alphalph}
\usepackage{chessfss}
\usepackage{tikz}
\usetikzlibrary[patterns]

\setboardfontencoding{LSBC5}

\usepackage[caption = false]{subfig}

\usepackage[longnamesfirst,numbers,sort&compress]{natbib}
\makeatletter
\def\NAT@spacechar{~}
\makeatother
\usepackage[foot]{amsaddr}

\usepackage{thmtools}
\usepackage{thm-restate}
\usepackage{hyperref}
\hypersetup{
colorlinks,
linkcolor={cDeepBlue},
citecolor={green!60!black},
urlcolor={blue!80!black},
pdftitle = {Constructions, bounds, and algorithms for peaceable queens},
pdfauthor = {}
}
\usepackage[compress, capitalise, nameinlink, noabbrev]{cleveref} 
\let\clearpage\relax

\declaretheorem[name = Theorem, numberwithin = section, style = plain, refname = {Theorem,Theorems}, Refname = {Theorem,Theorems}]{theorem}

\declaretheorem[name = Corollary, numberlike = theorem, style = plain, refname = {Corollary,Corollaries}, Refname = {Corollary,Corollaries}]{corollary}

\declaretheorem[name = Lemma, numberlike = theorem, style = plain, refname = {Lemma,Lemmas}, Refname = {Lemma,Lemmas}]{lemma}

\declaretheorem[name = Conjecture, numberlike = theorem, style = plain, refname = {Conjecture,Conjectures}, Refname = {Conjecture, Conjetures}]{conjecture}

\newcommand{\N}{\mathbb{N}}

\newcommand{\Z}{\mathbb{Z}}
\newcommand{\T}{\mathbb{T}}
\newcommand{\G}{\mathbb{G}}

\newcommand{\defn}[1]{\textcolor{cMaroon}{\emph{#1}}}

\renewcommand{\geq}{\geqslant}
\renewcommand{\leq}{\leqslant}

\DeclarePairedDelimiter{\ceil}{\lceil}{\rceil}
\DeclarePairedDelimiter{\floor}{\lfloor}{\rfloor}

\definecolor{cMaroon}{HTML}{93152a}
\definecolor{cYellow}{HTML}{f0e442}
\definecolor{cOrange}{HTML}{f7903b}
\definecolor{cRed}{HTML}{ff3341}
\definecolor{cPurple}{HTML}{da58c2}
\definecolor{cIndigo}{HTML}{8a58da}
\definecolor{cDeepBlue}{HTML}{4a8ce8}
\definecolor{cLightBlue}{HTML}{43e0ef}
\definecolor{cGreen}{HTML}{3ef450}

\tikzset{not_vert/.style={
    fill=white,
    text=black,
    draw opacity=0,
    minimum width=0
}}

\newcommand{\CartProd}{\mathbin{\square}}

\allowdisplaybreaks
\begin{document}

\title{Constructions, bounds, and algorithms for peaceable queens} 
\author{Katie Clinch}
\address[Katie Clinch]{School of Computer Science and Engineering \\ UNSW \\ Australia}
\email{k.clinch@unsw.edu.au} 

\author{Matthew Drescher}
\address[Matthew Drescher]{Department of Electrical and Computer Engineering \\ University of California, Davis \\ USA}
\email{knavely@gmail.com} 

\author{Tony Huynh}
\address[Tony Huynh]{Department of Computer Science \\ Sapienza Università di Roma \\ Italy}
\email{tony.bourbaki@gmail.com}

\author{Abdallah Saffidine}
\address[Abdallah Saffidine]{School of Computer Science and Engineering \\ UNSW \\ Australia}
\email{abdallah.saffidine@gmail.com}




\begin{abstract} 
The \defn{peaceable queens problem} asks to determine the maximum number $a(n)$ such that there is a placement of $a(n)$ white queens and $a(n)$ black queens on an $n \times n$ chessboard so that no queen can capture any queen of the opposite color.

In this paper, we consider the peaceable queens problem and its variant on the toroidal board.  For the regular board, we show that $a(n) \leq 0.1716n^2$, for all sufficiently large $n$.  This improves on the bound $a(n) \leq 0.25n^2$ of van Bommel and MacEachern~\cite{BM18}.  

For the toroidal board, we provide new upper and lower bounds. 
Somewhat surprisingly, our bounds show that there is a sharp contrast in behaviour between the odd torus and the even torus.  Our lower bounds are given by explicit constructions.  For the upper bounds, we formulate the problem as a non-linear optimization problem with at most $100$ variables, regardless of the size of the board. We solve our non-linear program exactly using modern optimization software.  

We also provide a local search algorithm and a software implementation which converges very rapidly to solutions which appear optimal.  Our algorithm is sufficiently robust that it works on both the regular and toroidal boards.  For example, for the regular board, the algorithm quickly finds the so-called \defn{Ainley construction}.  Thus, our work provides some further evidence that the Ainley construction is indeed optimal.  
\end{abstract}
\maketitle


\section{Introduction}

The \defn{peaceable queens problem} asks to determine the maximum number $a(n)$ such that there is a placement of $a(n)$ white queens and $a(n)$ black queens on an $n \times n$ chessboard so that no queen can capture any queen of the opposite color.
It was posed by Bosch~\cite{bosch1999} in 1999, and was added to the Online Encyclopedia of Integer Sequences (OEIS) in 2014 by Donald E. Knuth.

Although very simple to state, the peaceable queens problem has proven to be an extremely difficult optimization problem~\cite{Barbara2004, kane2017, Petrie2002, NW2013, LS14} with surprisingly varied and aesthetic optimal solutions.  Currently, only the first 15 terms of $a(n)$ are known \cite{oeis2500}.  By a popular vote of the editors of the Online Encyclopedia of Integer Sequences (OEIS), the sequence $a(n)$ was awarded notable sequence number A250000. The problem was popularized by N.J.A. Sloane in his \emph{Notices of the AMS} article~\cite{sloane18}, as well as his delightful Numberphile video on YouTube~\cite{sloanepeaceableVid}. 

The earliest known reference for the peaceable queens problem is actually Stephen Ainley's 1977 book \emph{Mathematical Puzzles}~\cite{ainley1977mathematical}. In his book, Ainley gives a construction which shows that $a(n) \geq \lfloor\frac{7}{48} n^2\rfloor \approx 0.1458 n^2$. 
\begin{figure}[h] 
\centering
\includegraphics[width=2.2in]{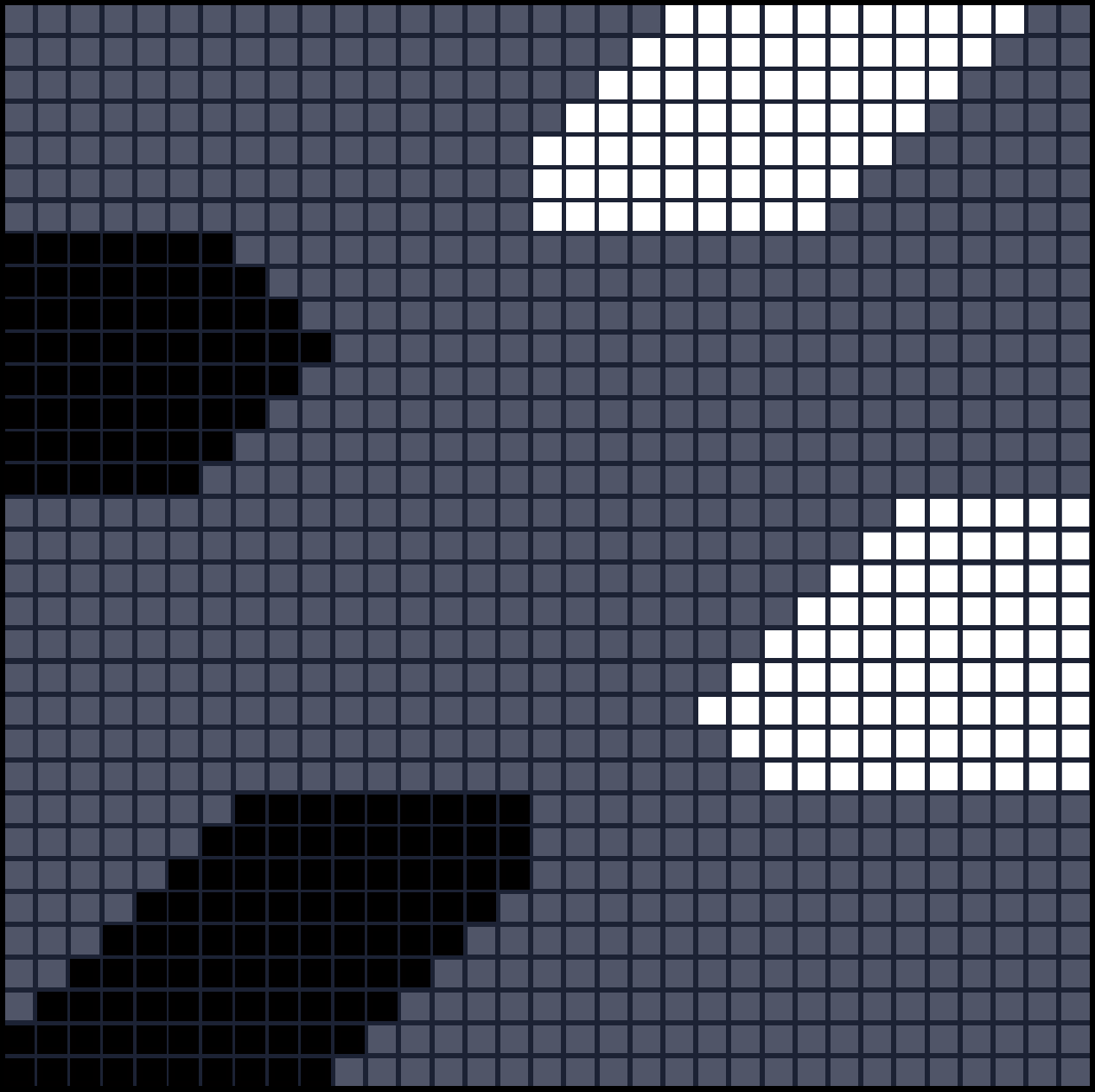}
\caption{Ainley's construction for $n=33$.  The number of white queens is $158=\lfloor \frac{7}{48}n^2 \rfloor$.}
\label{fig:ainley}
\end{figure}
This construction (see Figure~\ref{fig:ainley}) was rediscovered by Benoit Jubin in 2015, and later Yao and Zeilberger conjectured that Ainley's construction is optimal~\cite{yao2022numerical}.  Our work also seems to support this conjecture.  

There are natural extensions of the peaceable queens problem to other ``surfaces'' by identifying the boundary of the board appropriately (see~\cite{harris2013bishop}).  Our methods can be easily adapted to work on all such surfaces.  However, for simplicity, we focus on the regular board and the toroidal board, whose corresponding sequences we denote by $a(n)$ and $t(n)$, respectively. 
 The sequence $t(n)$ also appears in the OEIS as sequence number A279405~\cite{oeis279}.   At the time of writing, only the first 12 terms of $t(n)$ are known, due to Zabolotskiy and  Huchala~\cite{oeis279}. 
However, with these first few terms it is already clear that the behaviour of $a(n)$ and $t(n)$ differ substantially. For example, it is obvious that $a(n)$ is non-decreasing since the $n \times n$ regular board embeds into the $(n+1) \times (n+1)$ regular board.  However, this is not true for the toroidal board, since $t(8)> t(9)$.
 
The following are our main results.  

\begin{theorem}[Even torus lower bound] \label{thm:eventoruslowerbound}
\[t(n) \geq 0.1339n^2,\] for all sufficiently large even $n$. 
 \end{theorem}

The lower bound in~Theorem~\ref{thm:eventoruslowerbound} is given by an explicit construction.  See~Figure~\ref{fig:eventoruslower}.  

\begin{figure}[h] 
\centering
\includegraphics[width=2.2in]{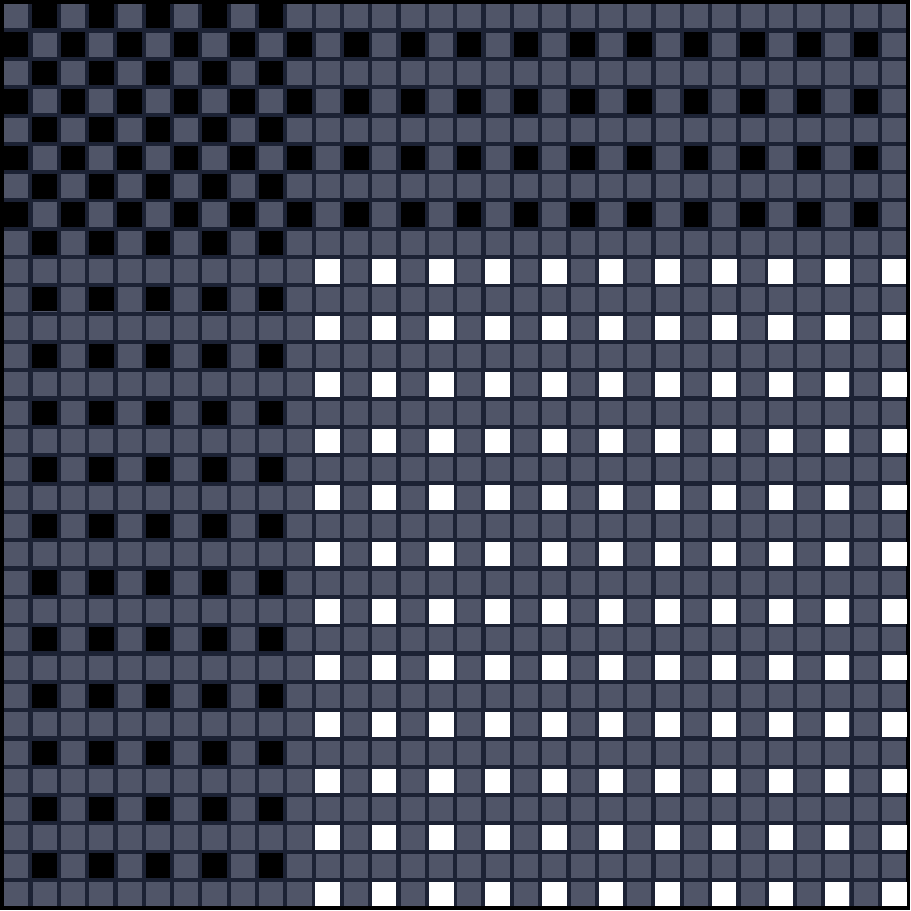}
\caption{Lower bound construction for the even torus.}
\label{fig:eventoruslower}
\end{figure}

\begin{theorem}[Even torus upper bound] \label{thm:eventorusupperbound}
\[t(n) \leq 0.1402n^2,\] for all even $n$.
\end{theorem}

\begin{theorem}[Odd torus lower bound] \label{thm:oddtoruslowerbound}
\[t(n) \geq 0.0833 n^2,\] for all sufficiently large odd $n$. 
\end{theorem}

The lower bound in~Theorem~\ref{thm:oddtoruslowerbound} is given by an explicit construction.  See~Figure~\ref{fig:Odd_Lower_ProofIdea}. 

\begin{figure}[ht]
    \centering
    \includegraphics[width=2.2in]{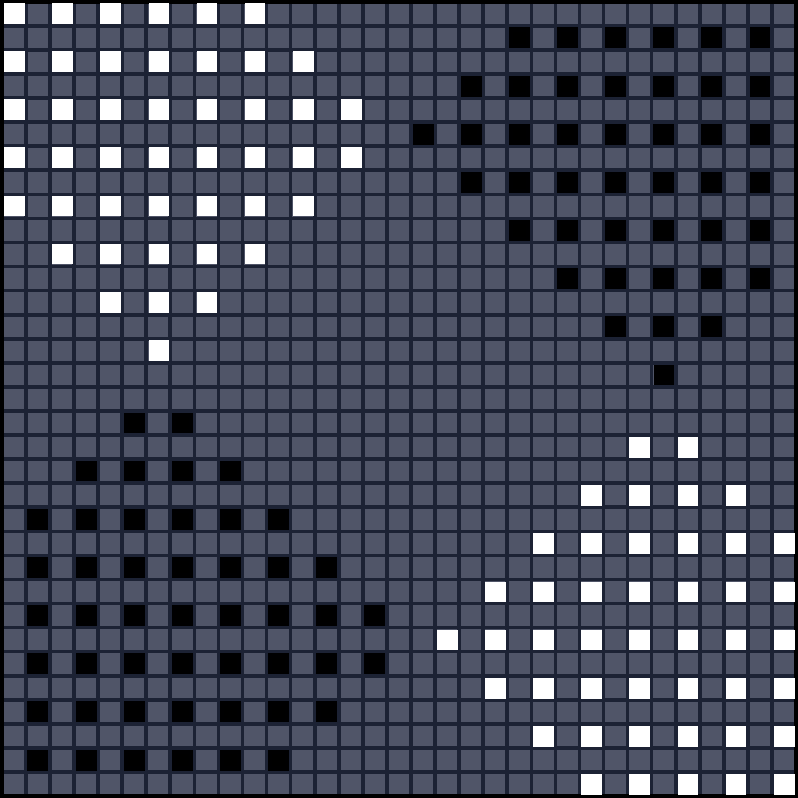}
    \caption{Lower bound construction for the odd torus.}
    \label{fig:Odd_Lower_ProofIdea}
\end{figure}

\begin{theorem}[Odd torus upper bound]\label{thm:oddtorusupperbound}
\[t(n) \leq 0.125n^2,\] for all odd $n$.  
\end{theorem}

\begin{theorem}[Regular board upper bound] \label{thm:regularboardupperbound}
$a(n) \leq 0.1716n^2$, for all sufficiently large $n$.
    \end{theorem}

Our upper bounds are obtained by solving a non-linear program which models a carefully chosen Venn diagram.  Roughly speaking, there is a variable for the size of each `region' of the Venn diagram.  We use a slightly different non-linear program for the even torus, odd torus, and regular board.  However, in all three cases, the number of variables is at most $100$, regardless of the board size. This allows us to solve all three non-linear programs exactly using modern optimization packages.  The astute reader will notice that our upper bound for the even torus is significantly better than our upper bounds for the odd torus and the regular board. This is because we encode extra symmetries into our non-linear program which do not exist for the odd torus and the regular board.  

Note that our lower bound for the even torus is much larger than our upper bound for the odd torus.  Therefore, the behaviour of $t(n)$ differs significantly for odd and even $n$.  In particular, we obtain the following corollary. 

\begin{corollary}
$t(2n)-t(2n-1) \geq 0.0356n^2$, for all sufficiently large $n$. 
\end{corollary}

Our bound $a(n) \leq 0.1716n^2$ from~Theorem~\ref{thm:regularboardupperbound} improves on the bound $a(n) \leq 0.25n^2$ proven by van Bommel and MacEachern~\cite{BM18}.  The proof from~\cite{BM18} does not actually use the fact that queens can attack diagonally.  In contrast, our non-linear program crucially exploits this fact.  This is necessary since there is a lower bound of $0.25n^2$ for the `peaceable rooks' problem. In principle, it may be possible to match the lower bound $\approx 0.1458 n^2$ given by the Ainley construction by adding more valid constraints to our non-linear program.

We also describe a simple local search algorithm that appears to converge very rapidly to optimal solutions for both the regular and toroidal boards.  For example, for the regular board, our algorithm quickly finds the Ainley example as well as many of the other examples listed on the OEIS which match the Ainley example.  For the even torus, our algorithm quickly finds the example from~Theorem~\ref{thm:eventoruslowerbound}, as well as other examples which match that bound. For the odd torus, our algorithm actually finds examples with significantly more queens than our construction from~Theorem~\ref{thm:oddtoruslowerbound}.  However, due to their irregularity, we were unable to generalize these examples to a construction that works for all odd $n$. Our experimental data suggests that the Ainley construction is optimal for the regular board, and the construction from~Theorem~\ref{thm:eventoruslowerbound} is (essentially) optimal for the even torus (see~Conjecture~\ref{conj:EvenAlmostOptimal}). We have embedded our algorithm into an interactive "browser app" called \texttt{\textcolor{cMaroon}{Pieceable Queens}}\footnote{We strongly recommend to run Pieceable Queens in Firefox.},  where the interested reader can run it for themselves~\cite{pieceable}.

\textbf{Paper Outline.}  We begin with some notation in~Section~\ref{sec:notation}. In Section~\ref{sec:lower}, we describe our explicit constructions which prove Theorems~\ref{thm:eventoruslowerbound}, and~\ref{thm:oddtoruslowerbound}.  In~Section~\ref{sec:Hal9000}, we define our non-linear programs and prove~Theorem~\ref{thm:eventorusupperbound},~Theorem~\ref{thm:oddtorusupperbound}, and~Theorem~\ref{thm:regularboardupperbound}.  We describe our algorithm in~Section~\ref{sec:algorithm}.  In~Appendix~\ref{sec:appendix} we list the best odd torus solutions found by our algorithm.  

\section{Notation} \label{sec:notation}

For integers $a \leq b$, we let $[a, b]:=\{a, \dots, b\}$ and $[a]=[1,a]$.  We refer to elements of $[n] \times [n]$ as \defn{cells}. We define $X \subseteq [n] \times [n]$ to be
\begin{itemize}
    \item 
    a \defn{row} if $X=\{i\} \times [n]$ for some $i \in [n]$,
    \item 
    a \defn{column} if $X=[n] \times \{j\}$ for some $j \in [n]$,
    \item 
    a \defn{diagonal} if $X=\{(i,j) \mid i-j=k\}$ for some $k \in [-(n-1), n-1]$,
    \item
    a \defn{skew-diagonal} if $X=\{(i,j) \mid i+j=k\}$ for some $k \in [2, 2n]$,
    \item
    a \defn{$\mathbb{Z}_n$-diagonal} if $X=\{(i,j) \mid i-j \equiv k \pmod{n} \}$ for some $k \in \mathbb{Z}_n$, and
    \item 
    a \defn{$\mathbb{Z}_n$-skew-diagonal} if $X=\{(i,j) \mid i+j\equiv k \pmod{n} \}$ for some $k \in \mathbb{Z}_n$.
\end{itemize}

The \defn{$n \times n$ grid}, denoted $\mathbb{G}_n$, is the hypergraph with vertex set $[n] \times [n]$, whose hyperedges are the rows, columns, diagonals, and skew-diagonals of $[n] \times [n]$.  The \defn{$n \times n$ torus}, denoted $\mathbb{T}_n$, is the hypergraph with vertex set $[n] \times [n]$, whose hyperedges are the rows, columns, $\mathbb{Z}_n$-diagonals, and $\mathbb{Z}_n$-skew-diagonals of $[n] \times [n]$. Thus $\G_n$ and $\T_n$ each have four \defn{types} of hyperedges.  

A \defn{peaceful battle on $\mathbb{G}_n$ (respectively, $\mathbb{T}_n$)} is a pair $(B,W)$, where $B$ and $W$ are disjoint subsets of $[n] \times [n]$ such that $e \cap B=\emptyset$ or $e \cap W=\emptyset$ for every hyperedge $e \in E(\mathbb{G}_n)$ (respectively, $E(\mathbb{T}_n)$).  We say that $m \in \mathbb{N}$ is \defn{peaceable on $\mathbb{G}_n$ (respectively, $\mathbb{T}_n$)} if there is a peaceful battle $(B,W)$ on $\mathbb{G}_n$ (respectively, $\mathbb{T}_n$) such that $|B|=|W|=m$.  With this terminology, $a(n)$ is the maximum integer $m$ such that $m$ is peaceable on $\mathbb{G}_n$, and $t(n)$ is the maximum integer $m$ such that $m$ is peaceable on $\mathbb{T}_n$.  Since every hyperedge of $\mathbb{G}_n$ is contained in a  hyperedge of $\mathbb{T}_n$, we clearly have $a(n) \geq t(n)$, for all $n$.  

\section{Lower bounds}\label{sec:lower}
 We begin with our construction for the even torus.  
Let $n \in \mathbb{N}$ and $a,b \in [n]$.  The \defn{$(a,b)$-plaid} is the pair $(B_1 \cup B_2 \cup B_3, W)$, where 
\begin{align*}
B_1 &=\{(i,j) \in [a] \times [b] : \text{$i+j$ is odd}\},  \\
B_2 &=  \{(i,j) \in [a] \times [b+1, 2n] : \text{$i$ is even and $j$ is odd}\}, \\
B_3 &=   \{(i,j) \in [a+1,2n] \times [b] : \text{$i$ is odd and $j$ is even}\},
\end{align*}
and $W$ is the set of all pairs $(i,j) \in [a+1, 2n] \times [b+1, 2n]$ with $i$ and $j$ both even.  
See~Figure~\ref{fig:eventoruslower} for a drawing of the $(8,10)$-plaid on $\T_{32}$.
\begin{lemma}
    For all $n \in \mathbb{N}$ and $a,b \in [n]$, the $(a,b)$-plaid is a peaceful battle on $\mathbb{T}_{2n}$.  
\end{lemma}

\begin{proof}
  Let $(B,W)$ be the $(a,b)$-plaid and let $e \in E(\T_{2n})$.  It is clear that $e \cap B = \emptyset$ or $e \cap W=\emptyset$ if $e$ is a row or column.  It remains to consider the case that $e$ is a $\Z_{2n}$-diagonal or a $\Z_{2n}$-skew-diagonal.  Say that $X \subseteq V(\T_{2n})$ is \defn{parity consistent} if $i+j \equiv i'+j' \pmod{2}$ for all $(i,j), (i',j') \in X$. Note that all $\Z_{2n}$-diagonals and $\Z_{2n}$-skew-diagonals are parity consistent.  Since $i+j$ is odd for all $(i,j) \in B$ and $i+j$ is even for all $(i,j) \in W$, it follows that $e \cap B = \emptyset$ or $e \cap W=\emptyset$ if $e$ is a $\Z_{2n}$-diagonal or a $\Z_{2n}$-skew-diagonal. 
\end{proof}

The key property of $\T_{2n}$ used in the above proof is that all $\Z_{2n}$-diagonals and $\Z_{2n}$-skew-diagonals are parity consistent.  Note that this property does not hold for the odd torus, and so a different construction is required.  The next lemma shows how to choose an $(a,b)$-plaid $(B,W)$ so that both $B$ and $W$ are large.  

\begin{lemma}\label{lem:Even_LowerBound}
    Let $c := 2-\sqrt{3}$, $n \in \N$ be even, $a:=\floor{cn}$, and $b:=\ceil{cn}$. Let $(B,W)$ be the $(a,b)$-plaid on $\T_n$.  Then $\min\{|B|,|W|\} \geq \frac{c}{2} n^2 - O(n)$. 
\end{lemma}

\begin{proof}
Observe that
\begin{align*}
|B| &\approx \frac{1}{2} ab + \frac{1}{4}a(n-b)+\frac{1}{4}b(n-a) \\
    &\approx \frac{c^2n^2}{2}+\frac{c(1-c)n^2}{2} \\
    &=\frac{c}{2} n^2.
\end{align*}
On the other hand,
\[
|W| \approx \frac{1}{4} (n-a)(n-b) \approx \frac{(1-c)^2}{4}n^2=\frac{c}{2} n^2,
\]
where the last equality is the only place in the proof where we use the explicit value of $c$.  Moreover, 
the above computations clearly hold up to an $O(n)$ error. Thus, $\min\{|B|,|W|\} \geq \frac{c}{2} n^2 - O(n)$, as required.
\end{proof}

We are now ready to prove~Theorem~\ref{thm:eventoruslowerbound}.

\begin{proof}
Let $c':=\frac{2-\sqrt{3}}{2}$.  By~Lemma~\ref{lem:Even_LowerBound}, for any $\epsilon > 0$ there exists a sufficiently large $N$ such that for every even $n \geq N$ there is an $(a,b)$-plaid $(B,W)$ on $\T_n$ with $\min\{|B|,|W|\} \geq (c'-\epsilon) n^2$.  Since $c' > 0.1339$, we are done.  
\end{proof}

We now turn our attention to the odd torus.  Let $n \in \N$ be odd.  Let $W$ be the set of pairs $(i,j) \in [n] \times [n]$ such that
\begin{itemize}
    \item $i$ and $j$ are both odd,
    \item $-\lfloor \frac{n}{3} \rfloor \leq i-j \leq \lfloor \frac{n}{3} \rfloor-1$, and
    \item $i+j \leq \lfloor \frac{2n}{3} \rfloor $ or $i+j \geq \lceil \frac{4n}{3} \rceil +1$.
\end{itemize}
Similarly, let $B$ be the set of pairs $(i,j) \in [n] \times [n]$ such that
\begin{itemize}
    \item $i$ and $j$ are both even,
    \item $\lceil \frac{2n}{3}+1 \rceil \leq i+j \leq \lceil \frac{4n}{3} \rceil$, and
    \item $i-j \geq -\lfloor \frac{n}{3} \rfloor+1$ or $i-j \geq \lfloor \frac{n}{3} \rfloor$.
\end{itemize}

We define the \defn{$n$-argyle} to be the pair $(B,W)$.  See~Figure~\ref{fig:Odd_Lower_ProofIdea} for a picture of the $31$-argyle.

\begin{lemma}
    For all odd $n$, the $n$-argyle is a peaceful battle on $\T_n$.  
\end{lemma}

\begin{proof}
Let $(B,W)$ be the $n$-argyle and let $e \in E(\T_{n})$.  It is clear that $e \cap B = \emptyset$ or $e \cap W=\emptyset$ if $e$ is a row or column.  It remains to consider the case that $e$ is a $\Z_{n}$-diagonal or a $\Z_{n}$-skew-diagonal.  Suppose $e$ is a $\Z_{n}$-diagonal.  Then there exists $k \in \Z_n$ such that $i-j \equiv k \pmod{n}$ for all $(i,j) \in e$.  We may assume that $k \in [-\lfloor \frac{n}{3} \rfloor, n-\lfloor \frac{n}{3} \rfloor-1]$.  If 
$k \in [-\lfloor \frac{n}{3} \rfloor, \lfloor \frac{n}{3} \rfloor-1]$, then $e \cap B=\emptyset$.  On the other hand, if $k \in [\lfloor \frac{n}{3} \rfloor, n-\lfloor \frac{n}{3} \rfloor-1]$, then $e \cap W=\emptyset$.  The case that $e$ is a $\Z_{n}$-skew-diagonal is analogous.  
\end{proof}

Since $\frac{1}{12}> 0.0833$, the following lemma immediately implies~Theorem~\ref{thm:oddtoruslowerbound}.

\begin{lemma}
Let $n \in \N$ be odd and let $(B,W)$ be the $n$-argyle.  Then $\min\{|B|,|W|\} \geq \frac{1}{12}n^2 - O(n)$.
\end{lemma}

\begin{proof}
We cut a square $S$ of side length $n$ into a set of 36 congruent triangles $T_1, \dots, T_{36}$ as follows.  First cut $S$ into 9 congruent squares, and then cut each of the smaller 9 squares into four congruent triangles by cutting along their two diagonals.  Next, we regard each cell $(i,j) \in [n] \times [n]$ as a $1 \times 1$ square contained in $S$.  

\begin{figure}[ht]
    \centering
    \includegraphics[width=2.2in]{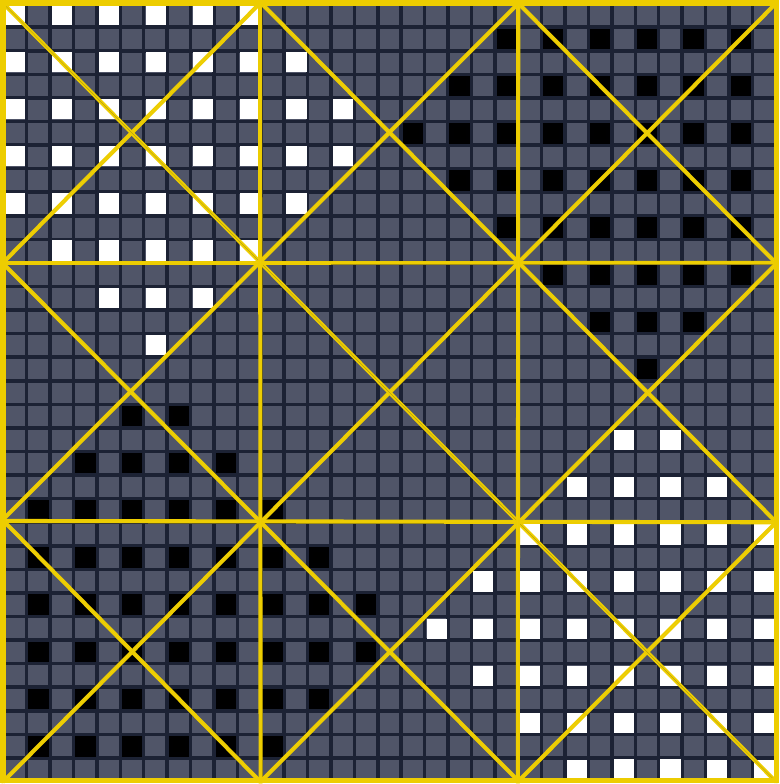}
    \caption{Partition of $[n] \times [n]$ into $36$ sets.}
    \label{fig:36triangles}
\end{figure}

See Figure~\ref{fig:36triangles} for a picture of the $33$-argyle embedded inside a square of side length $33$. For each cell $x$, we choose $f(x) \in [36]$ such that $x \cap T_{f(x)} \neq \emptyset$.  Note that the choice of $f(x)$ is unique for all but at most $10n$ cells of $[n] \times [n]$.  
Let $T_1', \dots, T_{36}'$ be the partition of $[n] \times [n]$, where $x \in T_{f(x)}'$ for all $x \in [n] \times [n] $.  
From Figure~\ref{fig:36triangles}, it is clear that up to an $O(n)$ error, $B$ intersects 12 of $T_1', \dots, T_{36}'$ with density $\frac{1}{4}$ and 24 of $T_1', \dots, T_{36}'$ with density 0. The same holds for $W$.  
Thus, up to an $O(n)$ error we have 
$
\min \{|B|, |W|\} \approx \frac{12}{36} \cdot \frac{1}{4} n^2 = \frac{1}{12} n^2.
$
\end{proof}

\section{Upper bounds} \label{sec:Hal9000}

In this section, we prove upper bounds on $a(n)$ and $t(n)$ via non-linear programming.  Our implementation of all three non-linear programs is available online~\cite{implementation}.  We begin with the odd torus.  
\subsection{Odd torus}
    Let $n \in \N$ be odd and let $(B,W)$ be a largest peaceful battle on $\T_n$ with $|B|=|W|$.   Suppose that $B$ intersects exactly $r$ rows, $c$ columns, $d$ $\Z_n$-diagonals and $s$ $\Z_n$-skew-diagonals of $\T_n$. Let $R,C,D,S \subseteq V(\T_n)$ be the union of these rows, columns, $\Z_n$-diagonals, and $\Z_n$-skew-diagonals, respectively.   
    Thus, $|R|=r n$, $|C|=c n$, $|D|=d n$ and $|S|=s n$. 
    Since $B$ is contained in $R \cap C \cap D \cap S$ and $W$ is disjoint from $R \cap C \cap D \cap S$, we have $|R\cap C\cap D\cap S| \geq |B|$ and $|\overline{R\cup C\cup D\cup S}| \geq |W|$.

    Observe that each hyperedge of $\T_n$ has $n$ vertices and if $e$ and $e'$ are distinct hyperedges of the same type, then $e \cap e'=\emptyset$.  On the other hand, if $e$ and $e'$ are of different types, then $|e \cap e'|=1$.  Thus, for all distinct $(x,X), (x',X')\in\{(r,R),(c,C),(d,D),(s,S)\} $ we have $|X\cap X'| = xx'$.

    We now introduce variables to encode the above equations.  
    Let $\mathcal{F}$ be the set of 16 regions of the Venn diagram of $R,C,D$, and $S$.   For each $F \in \mathcal{F}$, we let $z_F:=\frac{|F|}{n^2}\in[0,1]$.   
     See Figure~\ref{fig:venn} for a depiction\footnote{Note that some $z_F$ variables are missing from Figure~\ref{fig:venn} since it is impossible to draw 16 non-empty regions if we use disks to represent the four sets.} of some of the variables $z_F$.

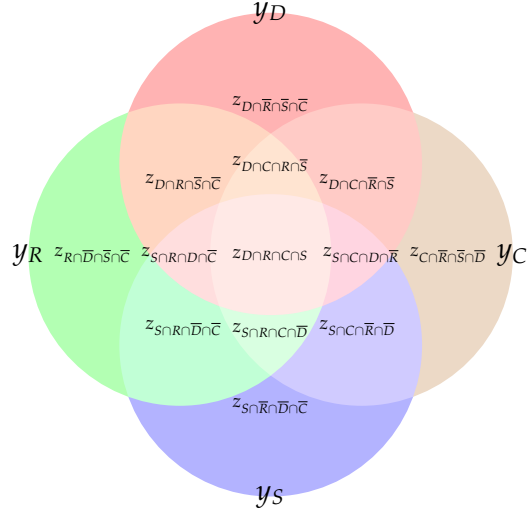
\begin{figure}[H]\label{fig:venn}
\subfloat{
\begin{tikzpicture}
  \begin{scope}[blend group = soft light]
    \fill[red!30!white]   ( 90:1.2) circle (2);
    \fill[green!30!white] (180:1.2) circle (2);
    \fill[blue!30!white]  (270:1.2) circle (2);
    \fill[brown!30!white]  (360:1.2) circle (2);
  \end{scope}
  \node at ( 90:3.2)    {$y_{D}$};
  \node at ( 180:3.2)    {$y_R$};
  \node at ( 270:3.2)    {$y_S$};
  \node at ( 360:3.2)    {$y_C$};
  \node at ( 90:2)    {\tiny{$z_{D \cap \overline{R} \cap \overline{S} \cap \overline{C}}$}};
  \node at ( 140:1.5)    {\tiny{$z_{D\cap R \cap \overline{S} \cap \overline{C}}$}};
  \node at ( 220:-1.5)    {\tiny{$z_{D\cap C \cap \overline{R} \cap \overline{S}}$}};
  \node at ( 270:-1.2)    {\tiny{$z_{D\cap C \cap R \cap \overline{S}}$}};
  \node at ( 270:2)   {\tiny{$z_{S \cap \overline{R} \cap \overline{D} \cap \overline{C}}$}};
  \node at ( 270:1)   {\tiny{$z_{S \cap R \cap C \cap \overline{D}}$}};
  \node at ( 320:1.5)   {\tiny{$z_{S \cap C \cap \overline{R} \cap \overline{D}}$}};
  \node at ( 220:1.5)   {\tiny{$z_{S \cap R \cap \overline{D} \cap \overline{C}}$}};
  \node at ( 360:1.2)   {\tiny{$z_{S \cap C \cap D \cap \overline{R}}$}};
  \node at ( 360:-1.2)   {\tiny{$z_{S \cap R\cap D \cap \overline{C}}$}};
  \node at ( 180:2.35)   {\tiny{$z_{R \cap \overline{D} \cap \overline{S} \cap \overline{C}}$}};
  \node at ( 360:2.35)   {\tiny{$z_{C \cap \overline{R} \cap \overline{S} \cap \overline{D}}$}};
  \node {\tiny{$z_{D \cap R \cap C \cap S}$}};
\end{tikzpicture}}
\caption{Each $y_X$ variable represents the fraction of vertices of $\T_n$ in $X$. Each $z_F$ variable represents the fraction of vertices of $\T_n$ contained in that region of the Venn Diagram.}
\end{figure}
    
    By the above equations, it suffices to solve the following optimization problem.

\begin{flalign*}
 &\textrm{maximize:} \min \{z_{R \cap C \cap D \cap S}, z_{\overline{R \cup C \cup D \cup S}}\} \\
 &\textrm{such that} \\
 &\sum_{F \in \mathcal{F}} z_F = 1 \\
 &y_{X} = \sum_{\substack{F \in \mathcal{F} \\ F \subseteq X}} z_F, \,  \text{for all $X \in \{R,C,D,S\}$} \\
  &y_{X \cap X'} = \sum_{\substack{F \in \mathcal{F} \\ F \subseteq X \cap X'}} z_F, \,  \text{for all $X,X' \in \{R,C,D,S\}$}\\
 & y_{X \cap X'} = y_Xy_{X'}, \,  \text{for all distinct $X,X' \in \{R,C,D,S\}$}\\
 &z_F  \in [0,1], \,  \text{for all $F \in \mathcal{F}$} 
\end{flalign*}

We ran interior point solver IPOPT, which used the exact Hessian method to find an optimal objective value of $\frac{1}{8}$ in 17 iterations~\cite{implementation}. It encodes the above model using the excellent Gekko optimization package~\cite{beal2018gekko}.  This proves~Theorem~\ref{thm:oddtorusupperbound}.

\subsection{Even torus}
 Let $n \in \N$ be even and let $(B,W)$ be a largest peaceful battle on $\T_n$ with $|B|=|W|$.  Our model is similar to the odd torus, except that on the even torus $\T_n$, there are
diagonals and skew-diagonals which do not intersect.  To deal with this, we refine the types of hyperedges as follows.  A $\Z_n$-diagonal $e$ is \defn{odd} (respectively, \defn{even}) if for all $(i,j) \in e$, $i-j$ is even (respectively, odd).  Similarly, a $\Z_n$-skew-diagonal $e$ is \defn{odd} (respectively, \defn{even}) if for all $(i,j) \in e$, $i+j$ is even (respectively, odd). Since all $Z_n$-diagonals are parity consistent, every $\Z_n$-diagonal is either even or odd.  The same holds for $\Z_n$-skew-diagonals.  

We now proceed as in the odd torus, except that there are more variables.  To be precise,  suppose that $B$ intersects exactly $r$ rows, $c$ columns, $d_0$ even $\Z_n$-diagonals, $d_1$ odd $\Z_n$-diagonals, $s_0$ even $\Z_n$-skew-diagonals, and $s_1$ odd $\Z_n$-skew-diagonals of $\T_n$. Let $R,C,D_0, D_1, S_0, S_1 \subseteq V(\T_n)$ be the union of these rows, columns, even $\Z_n$-diagonals, odd $\Z_n$-diagonals, even $\Z_n$-skew-diagonals, and odd $\Z_n$-skew-diagonals, respectively.  Thus, $|R|=r n$, $|C|=c n$, $|D_0|=d_0 n$, $|D_1|=d_1n$, $|S_0|=s_0 n$, and $|S_1|=s_1n$. Observe that

    \[
    |R\cap C\cap D_0\cap S_0|+ |R\cap C\cap D_1\cap S_1| \geq |B|,
    \]
    and
    \[
    |\overline{R\cup C\cup D_0 \cup D_1 \cup S_0 \cup S_1}| \geq |W|.
    \]

Let $e \in E(\T_n)$ and $e' \in E(\T_n)$.  If $e$ is a row and $e'$ is a column, then $|e \cap e'|=1$.  If $e$ is an even (respectively, odd) $\Z_n$-diagonal and $e'$ is an even (respectively, odd) $\Z_n$-skew-diagonal, then $|e \cap e'|=2$.  If $e$ is a row or column and $e'$ is a $\Z_n$-diagonal or a $\Z_n$-skew-diagonal, then $|e \cap e'|=1$.  Finally, if $e$ is an odd $\Z_n$-diagonal (respectively, odd $\Z_n$-skew-diagonal) and $e'$ is an even $\Z_n$-diagonal (respectively, even $\Z_n$-skew-diagonal), then $|e \cap e'|=0$.  Thus, $|R\cap C| = rc $, 
        $|D_0 \cap S_0|=2d_0s_0$, 
        $|D_1 \cap S_1|=2d_1s_1$; and $|X\cap X'| = 0$ if $(x,X) \in\{D_0,S_0\}$ and $X'\in \{D_1, S_1\}$; 
        and  $|X\cap X'| = xx'$ if $(x,X) \in\{(r,R),(c,C)\}$ and $(x',X') \in \{(d_0,D_0), (d_1, D_1), (s_0,S_0), (s_1, S_1)\}$.
       
 We now introduce variables to encode the above equations.  
    Let $\mathcal{F}$ be the set of 64 regions of the Venn diagram of $R,C,D_0, D_1, S_0$, and $S_1$.  For each $F \in \mathcal{F}$, we let $z_F:=\frac{|F|}{n^2}\in[0,1]$.  
    Let $\CartProd:=\{R,C\}, \times:=\{D_0, D_1, S_0, S_1\}$, and $\boxtimes:= \CartProd \cup \times$.  Let $F_1:=\overline{R \cup C \cup D \cup A}$, $F_2:=R \cap C \cap D_0 \cap \overline{D_1} \cap S_0 \cap \overline{S_1}$, and $F_3:=R \cap C \cap \overline{D_0} \cap D_1 \cap \overline{S_0} \cap S_1$. By the above equations, it suffices to solve the following optimization problem.

\begin{flalign*}
&\textrm{maximize:} \min \{z_{F_1}, z_{F_2} +  z_{F_3}\} \\
&\textrm{such that} \\
&\sum_{F \in \mathcal{F}} z_F = 1 \\
&y_{X} := \sum_{\substack{F \in \mathcal{F} \\ F \subseteq X}} z_F, \text{ for all $X \in \boxtimes$} \\
&y_{X \cap X'} := \sum_{\substack{F \in \mathcal{F} \\ F \subseteq X \cap X'}} z_F, \text{ for all  $X,X' \in \boxtimes$} \\
 &y_{R \cap C} = y_{R} y_C \\
 &y_{D_0 \cap S_0} = 2y_{D_0} y_{S_0} \\
  &y_{D_1 \cap S_1} = 2y_{D_1} y_{S_1} \\
    &y_{X \cap X'} = y_{X'} y_X, \text{ if $X \in \CartProd$ and $X' \in \times$}\\
    &y_{X \cap X'} = 0, \text{ if $X \in \{D_0,S_0\}$ and $X' \in \{D_1, S_1\}$}\\
 &y_X \in [0,\frac{1}{2}], \text{ for all $X \in \times$}  \\
 &z_F \in [0,1], \text{ for all $F \in \mathcal{F}$} \\
\end{flalign*}
Again, we ran interior point solver IPOPT via Gekko to find an optimal objective value of \texttt{0.140132093206979} in 100 iterations~\cite{implementation}. This proves~Theorem~\ref{thm:eventorusupperbound}.

\subsection{Regular Board}
Let $(B,W)$ be a largest peaceful battle on $\G_n$ with $|B|=|W|$.  We proceed as in the even torus, except with a minor adjustment since diagonals can have different lengths.  As in the even torus, we partition the set of diagonals into \defn{odd} or \defn{even} diagonals, and the set of skew-diagonals into \defn{odd} or \defn{even} skew-diagonals.  Let $R,C,D_0, D_1, S_0, S_1 \subseteq V(\G_n)$ be the union of these rows, columns, even diagonals, odd diagonals, even skew-diagonals, and odd skew-diagonals, respectively.  Let $\CartProd:=\{R,C\}, \times:=\{D_0, D_1, S_0, S_1\}$, and $\boxtimes:= \CartProd \cup \times$.  

    Let $\mathcal{F}$ be the set of 64 regions of the Venn diagram of $R,C,D_0, D_1, S_0$, and $S_1$.  For each $F \in \mathcal{F}$, we let $z_F:=\frac{|F|}{n^2}\in[0,1]$.   
Again we have
  \[
    |R\cap C\cap D_0\cap S_0|+ |R\cap C\cap D_1\cap S_1| \geq |B|,
    \]
    and
    \[
    |\overline{R\cup C\cup D_0 \cup D_1 \cup S_0 \cup S_1}| \geq |W|.
    \]

The next lemma provides the key constraint of our non-linear program for the regular board $\G_n$.

\begin{lemma} \label{lem:crazyconstraint}
Let $\alpha, \beta \in [0,1]$ satisfy $\alpha \leq 2 \sqrt{\beta}$.   
If $X \in \CartProd$, $X' \in \times$, $|X| \leq \alpha n^2$, and $|X'| \leq \beta n^2$, then  $|X \cap X'| \leq \alpha \sqrt{\beta}n^2-\frac{\alpha^2}{4}n^2+O(n)$.
\end{lemma}

\begin{proof}
Suppose $\lfloor \alpha n \rfloor$ is odd (the case that $\lfloor \alpha n \rfloor$ is even is similar).  Let $X \in \CartProd$, $X' \in \times$ with $|X| \leq \alpha n^2$ and $|X'| \leq \beta n^2$.  Let $E_{X'}$ be the hyperedges of type $X'$ contained in $X'$.  For each $e \in E_{X'}$ define $r(e):=\frac{|e \cap X|}{|e|}$.  Sort $E_{X'}$ in decreasing order according to $r(e)$. Let $\ell$ be the number of hyperedges of $E_{X'}$ with $r(e)=1$.  Observe that $\ell \leq \lceil \frac{\alpha n}{2} \rceil$.  Moreover, for each $i \in \N$, the $(\ell+i)$th hyperedge in this ordering satisfies $r(e) \leq \frac{\lfloor \alpha n \rfloor}{\lfloor \alpha n \rfloor+2i}$.  Let $Y$ consist of the first $\lfloor \alpha n \rfloor$ rows, and $Y'$ consist of all cells $(i,j)$ such that $j-i \in \{0,2,4, \dots , 2k \}$, where $k$ is the largest integer strictly less than $\lfloor \sqrt{\beta}n \rfloor$. Note that $|Y|=n\lfloor \alpha n \rfloor \leq \alpha n^2$ and $|Y'| \leq \frac{(2k)^2}{4} \leq \frac{(2\sqrt{\beta}n)^2}{4}=\beta n^2$. Let $E_{Y'}$ be the even anti-diagonals contained in $Y'$.  For each $e \in E_{Y'}$ define $r(e):=\frac{|e \cap Y|}{|e|}$.  Sort $E_{Y'}$ in decreasing order according to $r(e)$.  Let $\ell'$ be the number of hyperedges of $E_{Y'}$ with $r(e)=1$.
Note that $\ell'=\lceil \frac{\alpha n}{2} \rceil \geq \ell$.  Moreover, for each $i \in \N$, the $(\ell+i)$th hyperedge in this ordering satisfies $r(e) = \frac{\lfloor \alpha n \rfloor}{\lfloor \alpha n \rfloor+2i}$.  Therefore, $|Y \cap Y'| \geq |X \cap X'|$.  We conclude the proof by observing that
$
|Y \cap Y'| \approx \frac{(\alpha n)(2 \sqrt{\beta} n)}{2} - \frac{(\alpha n)^2}{4}=\alpha \sqrt{\beta}n^2-\frac{\alpha^2}{4}n^2,
$
where the above approximation clearly holds up to a $O(n)$ error. 
\end{proof}
     
     Let $F_1:=\overline{R \cup C \cup D \cup A}$, $F_2:=R \cap C \cap D_0 \cap \overline{D_1} \cap S_0 \cap \overline{S_1}$, and $F_3:=R \cap C \cap \overline{D_0} \cap D_1 \cap \overline{S_0} \cap S_1$.  Consider the following optimization problem.

\label{model:regular}
\begin{flalign*}
 &\textrm{maximize:} \min \{z_{F_1}, z_{F_2} +  z_{F_3}\} \\
 &\textrm{such that} \\
 &\sum_{F \in \mathcal{F}} z_F = 1 \\
 &y_{X} = \sum_{\substack{F \in \mathcal{F} \\ F \subseteq X}} z_F, \text{ for all $X \in \boxtimes$} \\
  &y_{X \cap X'} = \sum_{\substack{F \in \mathcal{F} \\ F \subseteq X \cap X'}} z_F, \text{ for all  $X,X' \in \boxtimes$} \\
   &y_{D_0 \cap D_1} = 0 \\
    &y_{D_0 \cap S_1} = 0 \\
    &y_{S_0 \cap D_1} = 0 \\
    &y_{S_0 \cap S_1} = 0 \\
  &y_{R \cap C} = y_{R} y_C \\
         &y_{X}^2+4y_{X \cap X'} \leq 4y_{X}\sqrt{y_{X'}}, \\
         &\text{if $X \in \CartProd$, $X' \in \times$, $y_X \leq 2 \sqrt{y_{X'}} \leq 1$}\\
 &y_X \in [0,\frac{1}{2}], \text{ for all $X \in \times$}  \\
 &z_F \in [0,1], \text{ for all $F \in \mathcal{F}$} \\
\end{flalign*}

Lemma~\ref{lem:crazyconstraint} immediately implies the following.  

\begin{lemma} \label{lem:alpha}
    If $\alpha \in [0,1]$ is the optimal value to the above optimization problem and $\alpha' > \alpha$ then $a(n) \leq \alpha' n^2$, for all sufficiently large $n$.
\end{lemma}

Here we used the APOPT solver via Gekko~\cite{beal2018gekko} which ran in 27 iterations and found an optimal solution 0.171572875253810.  This proves~Theorem~\ref{thm:regularboardupperbound}.  

\section{The Algorithm} \label{sec:algorithm}

We now describe our local search algorithm~\cite{pieceable}.  For concreteness, we will use the regular board $\G_n$, but the same algorithm works on the torus simply by replacing $\G_n$ by $\T_n$ in what follows.  A key concept is the notion of a \defn{swap}, which we now define.  For each $X \subseteq V(\G_n)$, let $E_X$ be the hyperedges $e \in E(\G_n)$ such that $e \cap X \neq \emptyset$, and let  $\widehat{X}:= V(\G_n) \setminus \bigcup_{e \in E_X} e$.  Note that $(X,\widehat{X})$ is a peaceful battle on $\G_n$ for all $X \subseteq V(\G_n)$.  Given a peaceful battle $(B,W)$ on $\G_n$ and $e \in E(\G_n)$, we say that $(B',W')$ is obtained from $(B,W)$ by \defn{swapping on $e$} if $B'=B \setminus e$ and $W'=\widehat{B'}$.  The main idea is that given a peaceful battle $(B,W)$ with $|B| > |W|$, we can attempt to use swaps to increase $\min \{|B|, |W|\}$.  For example,~Figure~\ref{fig:swapwindow} shows how to improve an $(a,b)$-plaid on $\T_{24}$ via swaps.  

\begin{figure}[ht]
\centering
\subfloat{\includegraphics[width=1.6in]{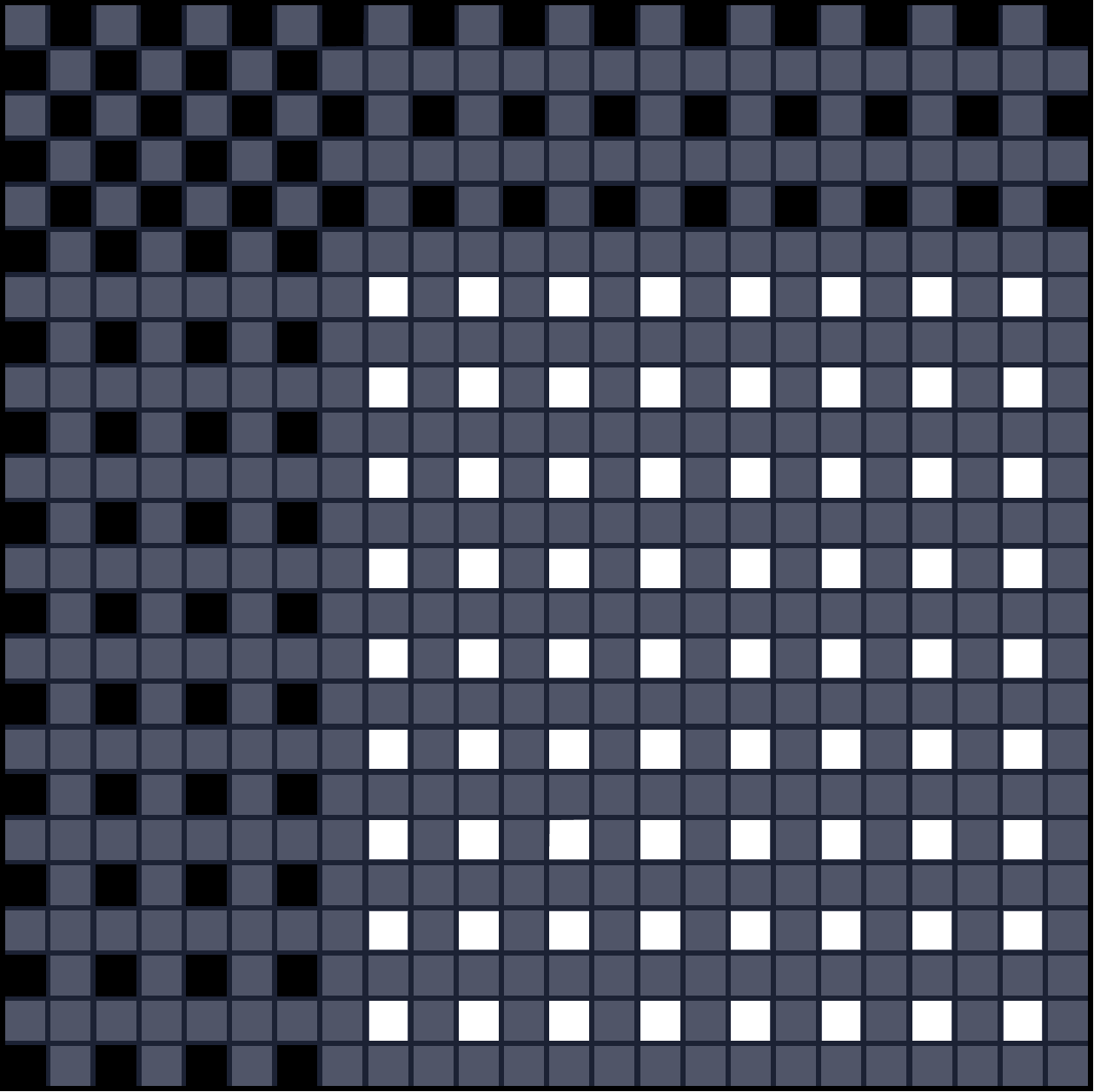}}\hfil 
\subfloat{\includegraphics[width=1.6in]{importantconfigurations/windowB_inkscape1.pdf}}
\caption{On the left is an $(a,b)$-plaid $(B,W)$ with $|B|=84$ and $|W|=72$. Let $c_{22}$ and $c_{24}$ be the last and third to last columns, and $r_{24}$ be the bottom row. By swapping on $c_{22}$, then $c_{24}$, and then $r_{24}$, we obtain the peaceful battle $(B', W')$ on the right.  Note that $|B'|=|W'|=74$, so $\min \{|B'|, |W'| \} > \min \{|B|, |W|\}$.}
\label{fig:swapwindow}
\end{figure}

Beginning with an initial peaceful battle $(B,W)$, the algorithm performs swaps to `improve' the solution until $\min \{|B|, |W| \} \geq q$, where $q$ is an input parameter.  
There are choices to be made about what `improve' means.  One natural candidate is to always improve the parameter $\min \{|B|, |W| \}$ at each step of the algorithm.  However, in practice we have found that it is advantageous to allow $\min \{|B|, |W| \}$ to decrease, provided $|B|+|W|$ increases.  Thus, every swap either increases $\min \{|B|, |W| \}$ or $|B|+|W|$.  It appears as if `most' initial choices of $(B,W)$ converge to an optimal solution via a sequence of swaps.  In practice, we have observed the fastest convergence by taking $|B|$ to be a random set of size around $\frac{n}{5}$ and $W$ to be $\widehat{B}$.  We also allow our algorithm to re-initialize itself using this initial condition. For a precise description of our algorithm, we refer the reader to~Algorithm~\ref{alg:main}.

\begin{algorithm}\caption{\sc{SWAP}}\label{alg:main} 
\begin{algorithmic}[1]
\REQUIRE $n$, target quality $q$ 
\ENSURE A feasible solution $(B,W)$ on $\G_n$ (or $\T_n)$.
\REPEAT 
\STATE{$B' \gets $ randomly select $n/5$ squares}
\STATE{$W' \gets \widehat{B'}$}
\REPEAT \label{line:repeat}
\STATE{$(B,W) \gets (B',W')$}
\STATE{$k \gets \min(|B|,|W|)$}
\IF{$|W| > |B|$}
    \STATE{$(B,W) \gets (W,B)$ \hfill\COMMENT{swap colors}}
\ENDIF  
\IF{there is a hyperedge $e$ such that $|\widehat{B \setminus e}| > |W| $}
\STATE{$W' \gets \widehat{B \setminus e}$; $B' \gets B \setminus e$ \hfill\COMMENT{swap on $e$ to improve $|W|$}}
\ENDIF            
\IF{there is a hyperedge $e$ such that $|B \setminus e|+|\widehat{B \setminus e}| > |B| + |W|$}
\STATE{$W' \gets \widehat{B \setminus e}$; $B' \gets B \setminus e$ \hfill\COMMENT{swap on $e$ to improve $|B| + |W|$}}
\ENDIF
\UNTIL{$\min(|B'|,|W'|) \leq k$}
\UNTIL{$|W| \geq q$}
\STATE{return $(B,W)$}\label{terminate}
\end{algorithmic}
\end{algorithm}

Our algorithm appears to perform extremely well in practice for both the regular and toroidal boards. 
For the odd torus $\T_n$, the algorithm seems to find solutions $(B,W)$ with $\min \{|B|, |W| \} \approx \frac{n^2}{11}$. This outperforms the $n$-argyle construction by $\approx 0.0075n^2$. In~Appendix~\ref{sec:appendix}, we list the best solutions output by our algorithm for $\T_n$ for all odd $n \in [13,63]$.  There does not appear to be a regular pattern among our examples.  It may be that for the odd torus, the optimal solution subtly depends on the prime factorization on $n$, in a way which we currently do not understand.  For instance,
 our example for $n=53$ actually contains fewer queens than our example for $n=51$. Besides this single exception, our solutions have a strictly increasing number of queens.

For the even torus, the algorithm quickly finds an $(a,b)$-plaid.  In fact, it is even able to \emph{improve} on an optimal $(a,b)$-plaid through swapping as seen in Figure~\ref{fig:swapwindow}. Therefore, there are optimal solutions for the even torus which are not $(a,b)$-plaids.  However, these `swapped' $(a,b)$-plaids have at most two more black queens than an optimal $(a,b)$-plaid.  This leads us to the following conjecture.  

\begin{conjecture}\label{conj:EvenAlmostOptimal}
   Let $n \in \N$ be even and let $(B,W)$ be an $(a,b)$-plaid on $\T_n$ with $\min \{|B|,|W|\}$ maximum. Then $t(n)\leq 2+ \min\{|B|,|W|\}$.
\end{conjecture}

Finally, for the regular board, our algorithm quickly finds the Ainley construction~\cite{ainley1977mathematical}, as well as other examples (some in the OEIS, some not) which meet the Ainley bound.  We did not manage to find any examples that beat the Ainley bound.  Thus, our experimental evidence suggests that the Ainley construction is indeed optimal.  
Although we do not have a proof of correctness, we conjecture that there exist $d \in \N$ and $p >0$  such that our algorithm outputs an optimal peaceful battle $(B,W)$ on $\G_n$ after at most $n^{d}$ steps with probability $p$.  

\section*{Acknowledgements} 
Matthew Drescher was supported by the National Science Foundation under Grant \#2127309 to the Computing Research Association for the CIFellows 2021 Project.  Thank you to Stephen Moehle for the excellent and thorough code reviews.  We are also grateful to N.J.A. Sloane for reading an early draft of this paper.  Matthew Drescher thanks John D. Owens and the Gunrock Graph GPU Lab for introducing him to the problem in a course assignment and for indulging in various interesting discussions.  Tony Huynh thanks the Institute for Basic Science in South Korea where part of this research was conducted.  
 
\bibliographystyle{DavidNatbibStyle}
\bibliography{references}
  \let\cleardoublepage\clearpage

 \cleardoublepage

    \appendix

  \section{Best odd torus solutions} \label{sec:appendix}
Below, we list the best peaceful battles found by~Algorithm~\ref{alg:main} on $\T_n$ for all odd $n$ between $13$ and $63$.  
 
 \begin{figure}[H]
\captionsetup[subfloat]{labelformat=empty}
\subfloat[$n=13, \min\{|B|, |W|\}=16$]{\includegraphics[width = \textwidth]{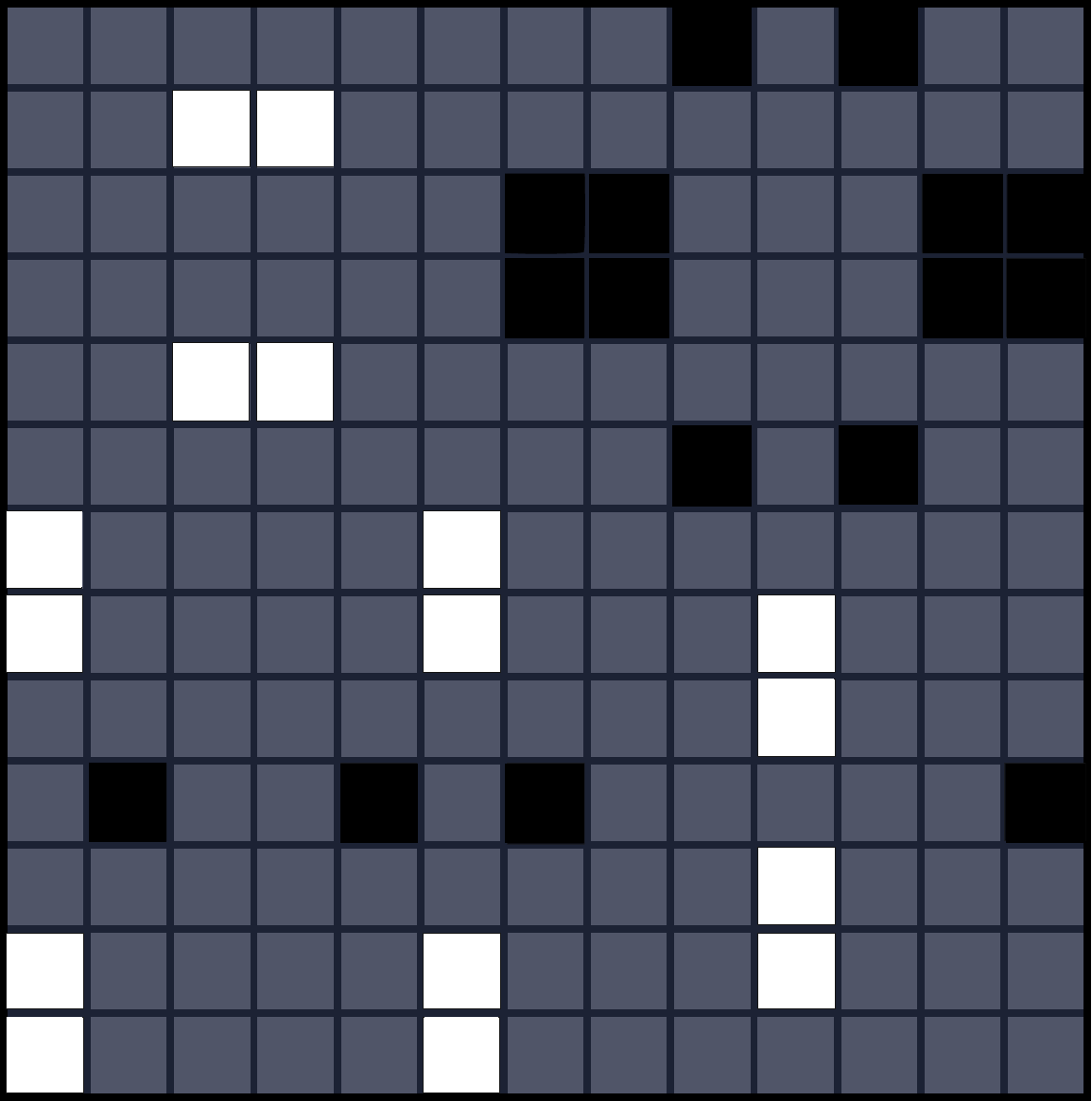}} 
\end{figure}

\begin{figure}[H]
\captionsetup[subfloat]{labelformat=empty}
\subfloat[$n=15, \min\{|B|, |W|\}=20$]{\includegraphics[width = \textwidth]{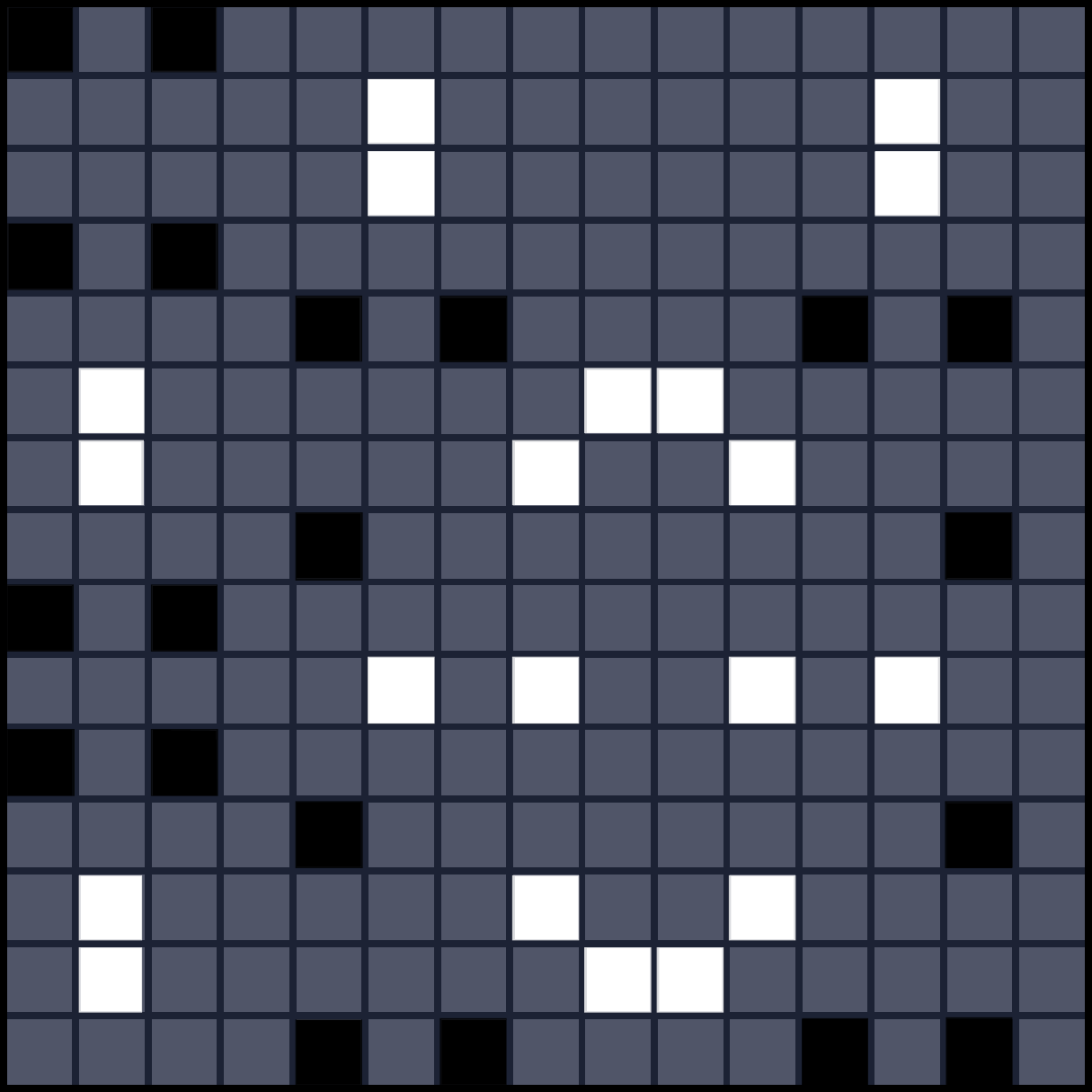}}
\end{figure}

\begin{figure}[H] 
\captionsetup[subfloat]{labelformat=empty}
\subfloat[$n=17, \min\{|B|, |W|\}=28$]{\includegraphics[width = \textwidth]{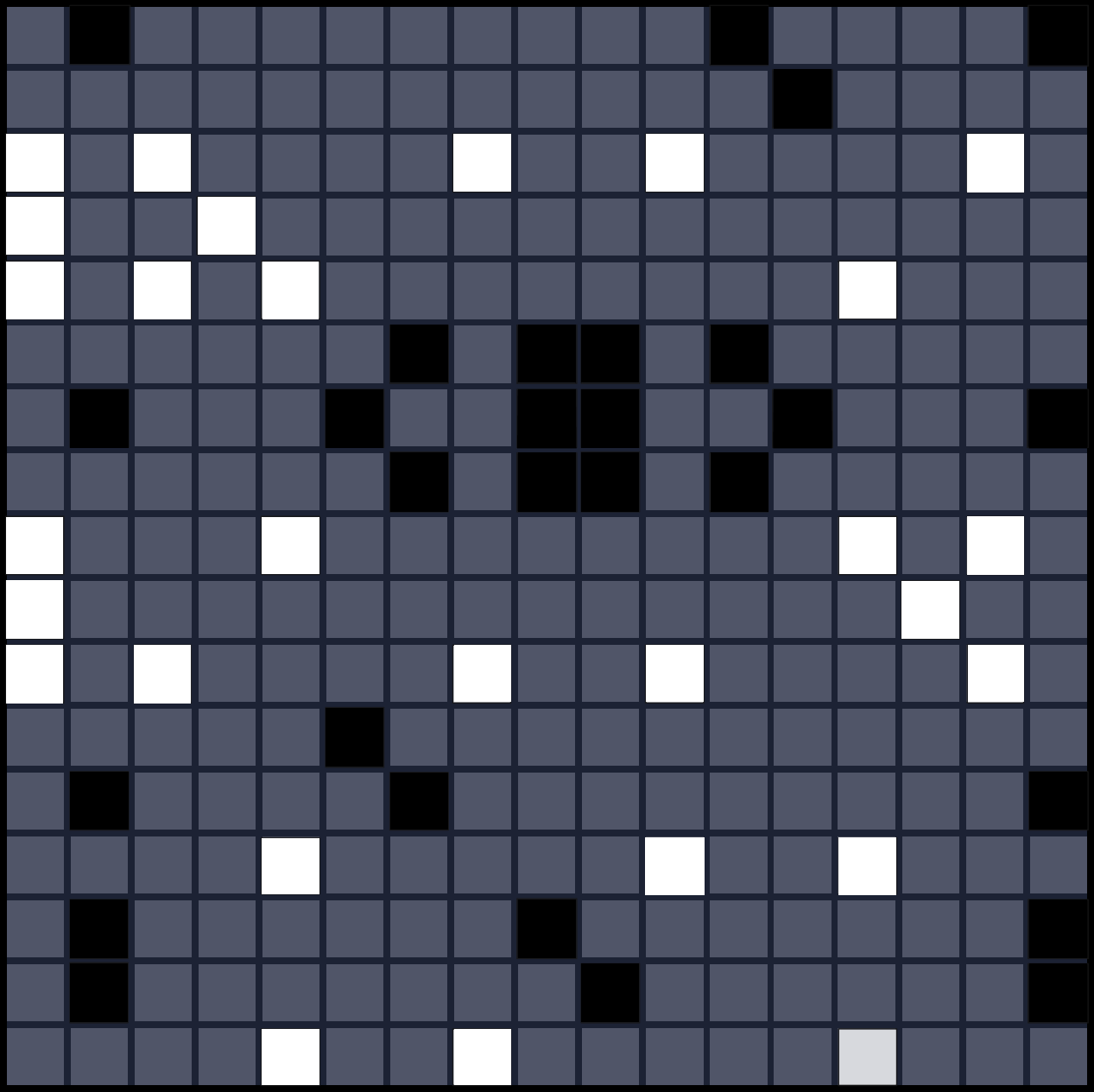}} 
\end{figure}

\begin{figure}[H]
\captionsetup[subfloat]{labelformat=empty}
\subfloat[$n=19, \min\{|B|, |W|\}=32$]{\includegraphics[width = \textwidth]{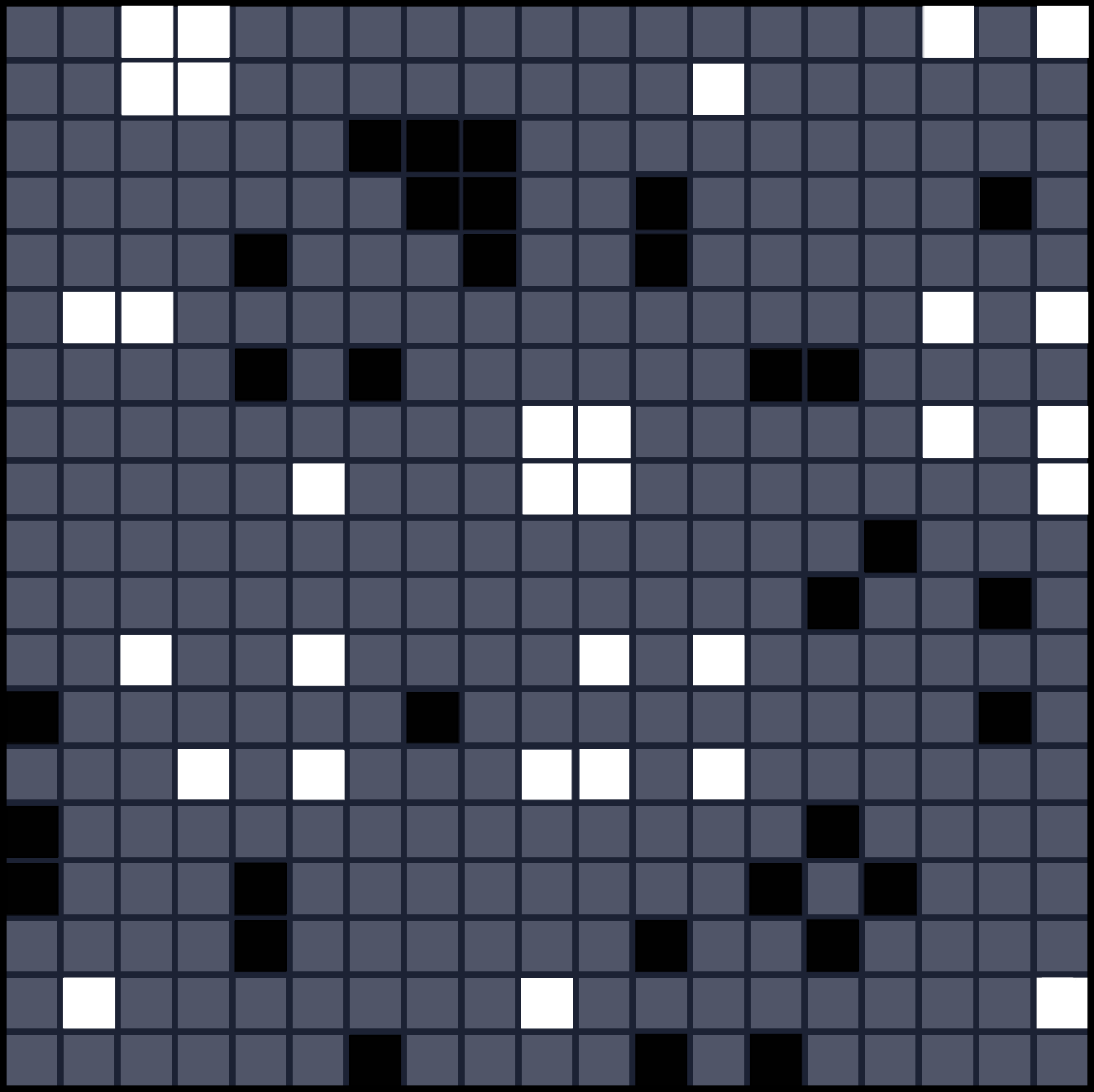}} 
\end{figure}

\begin{figure}[H]
\captionsetup[subfloat]{labelformat=empty}
\subfloat[$n=21, \min\{|B|, |W|\}=40$]{\includegraphics[width = \textwidth]{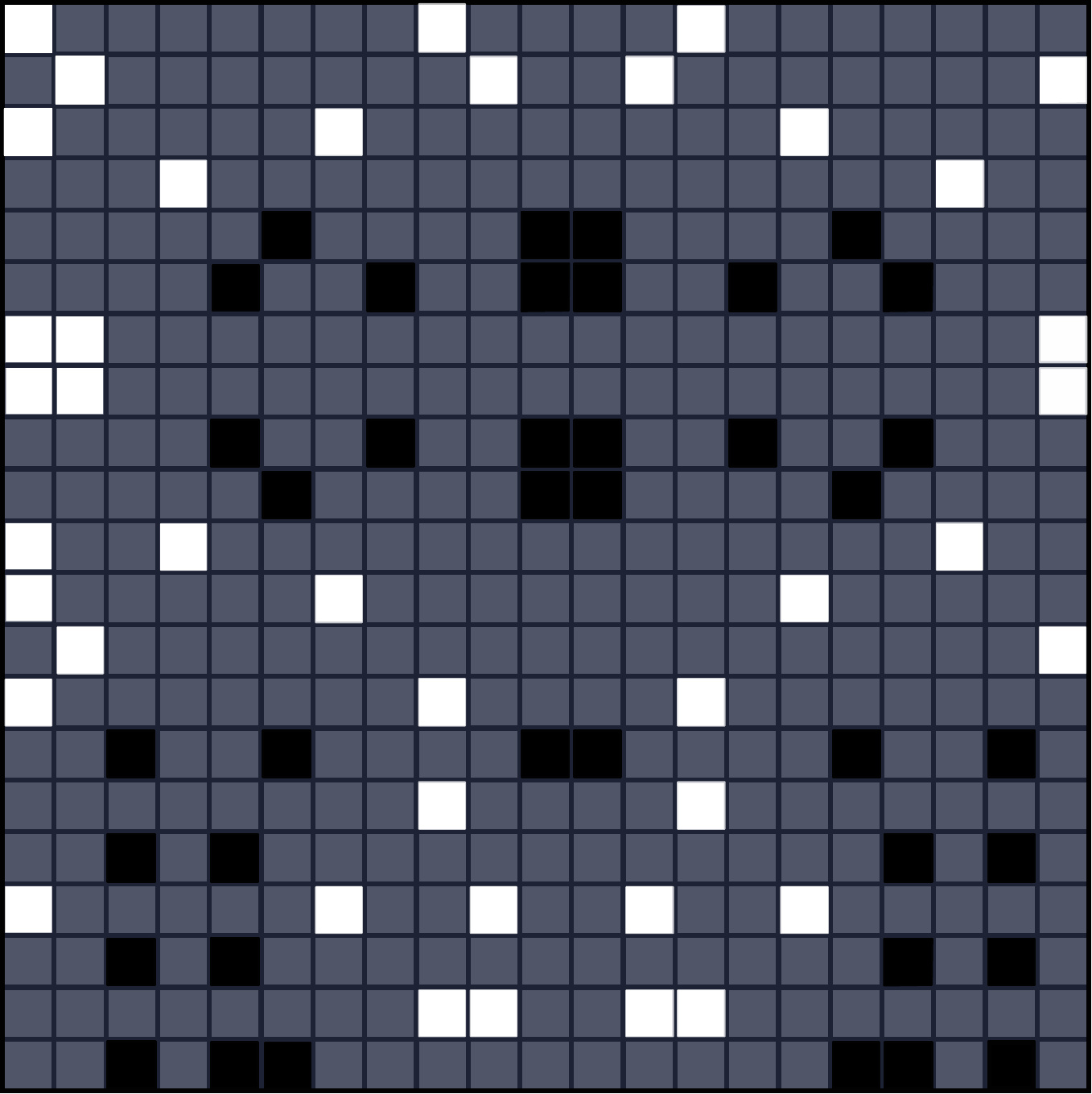}} 
\end{figure}

\begin{figure}[H]
\captionsetup[subfloat]{labelformat=empty}
\subfloat[$n=23, \min\{|B|, |W|\}=48$]{\includegraphics[width = \textwidth]{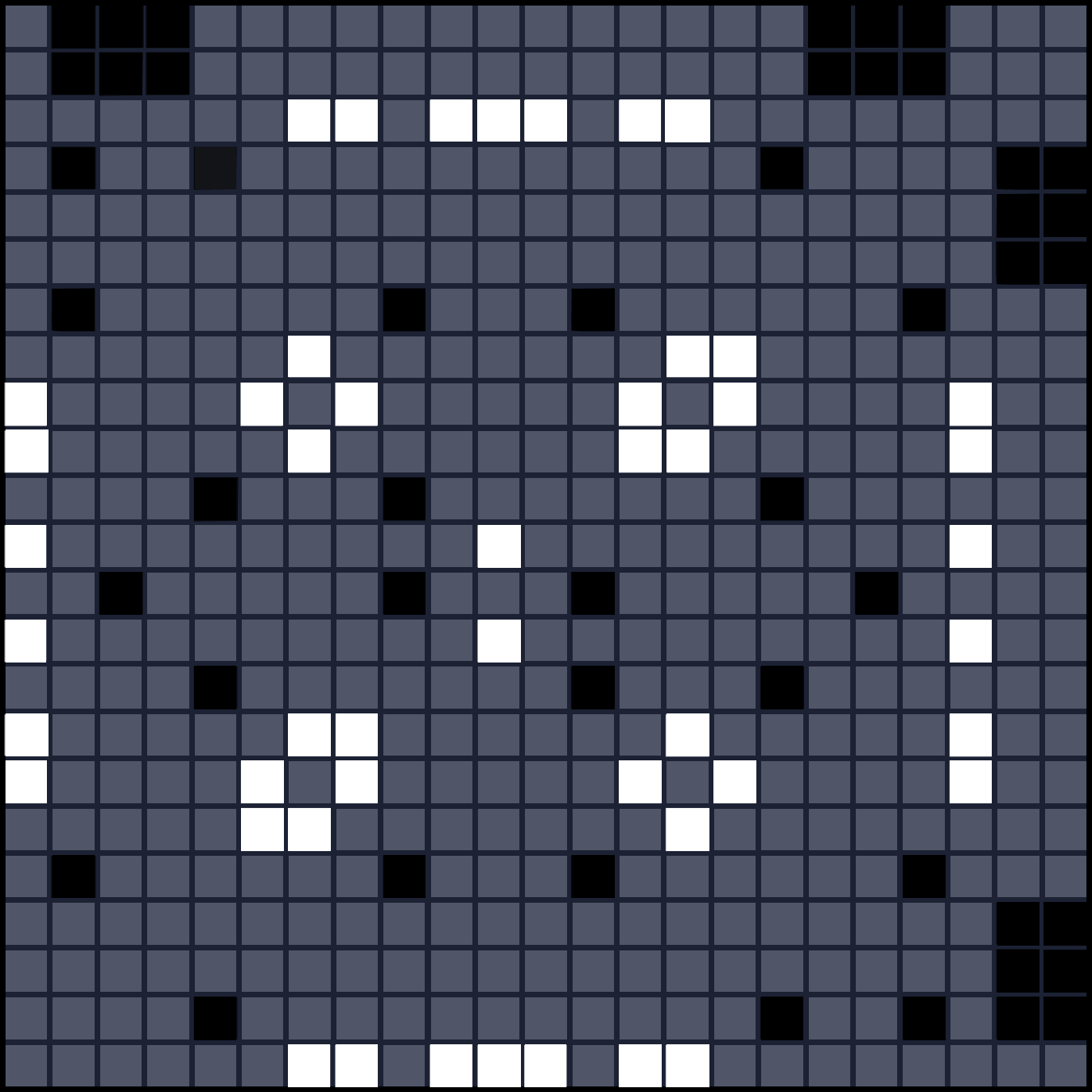}} 
\end{figure} 

\begin{figure}[H]
\captionsetup[subfloat]{labelformat=empty}
\subfloat[$n=25, \min\{|B|, |W|\}=56$]{\includegraphics[width = \textwidth]{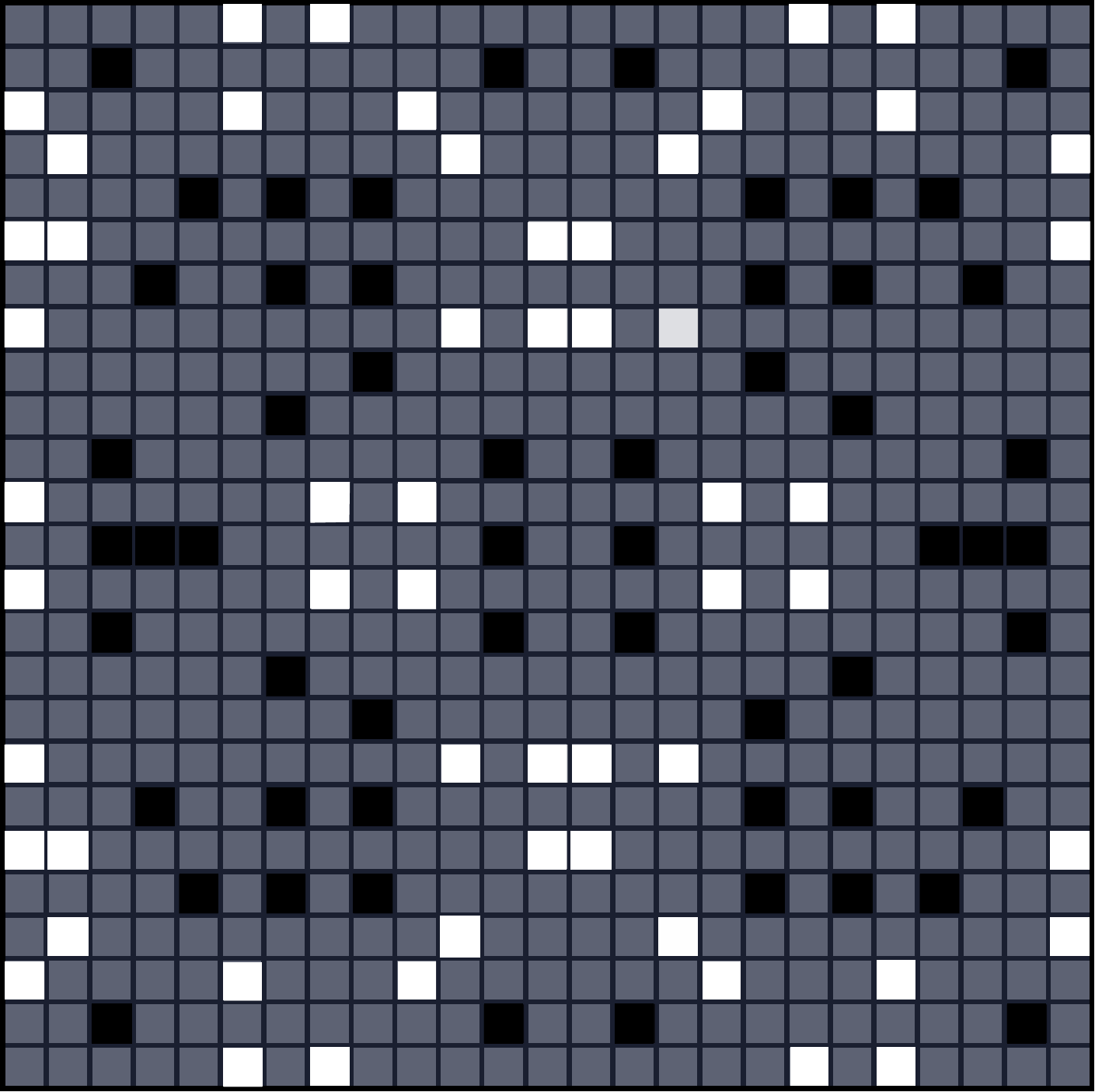}} \\
\end{figure}

\begin{figure}[H]
\captionsetup[subfloat]{labelformat=empty}
\subfloat[$n=27, \min\{|B|, |W|\}=66$]{\includegraphics[width = \textwidth]{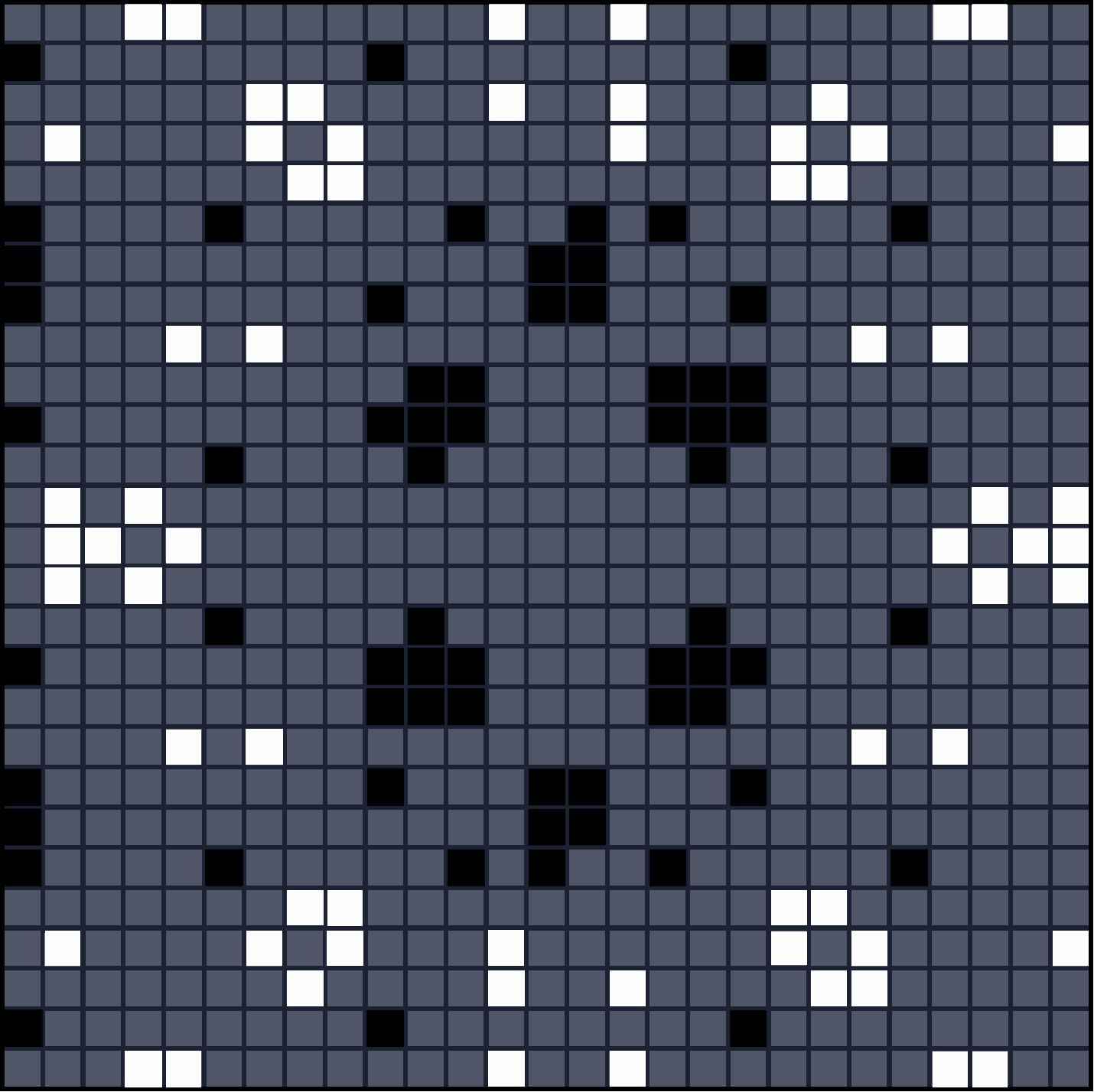}} 
\end{figure}

\begin{figure}[H]
\captionsetup[subfloat]{labelformat=empty}
\subfloat[$n=29, \min\{|B|, |W|\}=76$]{\includegraphics[width = \textwidth]{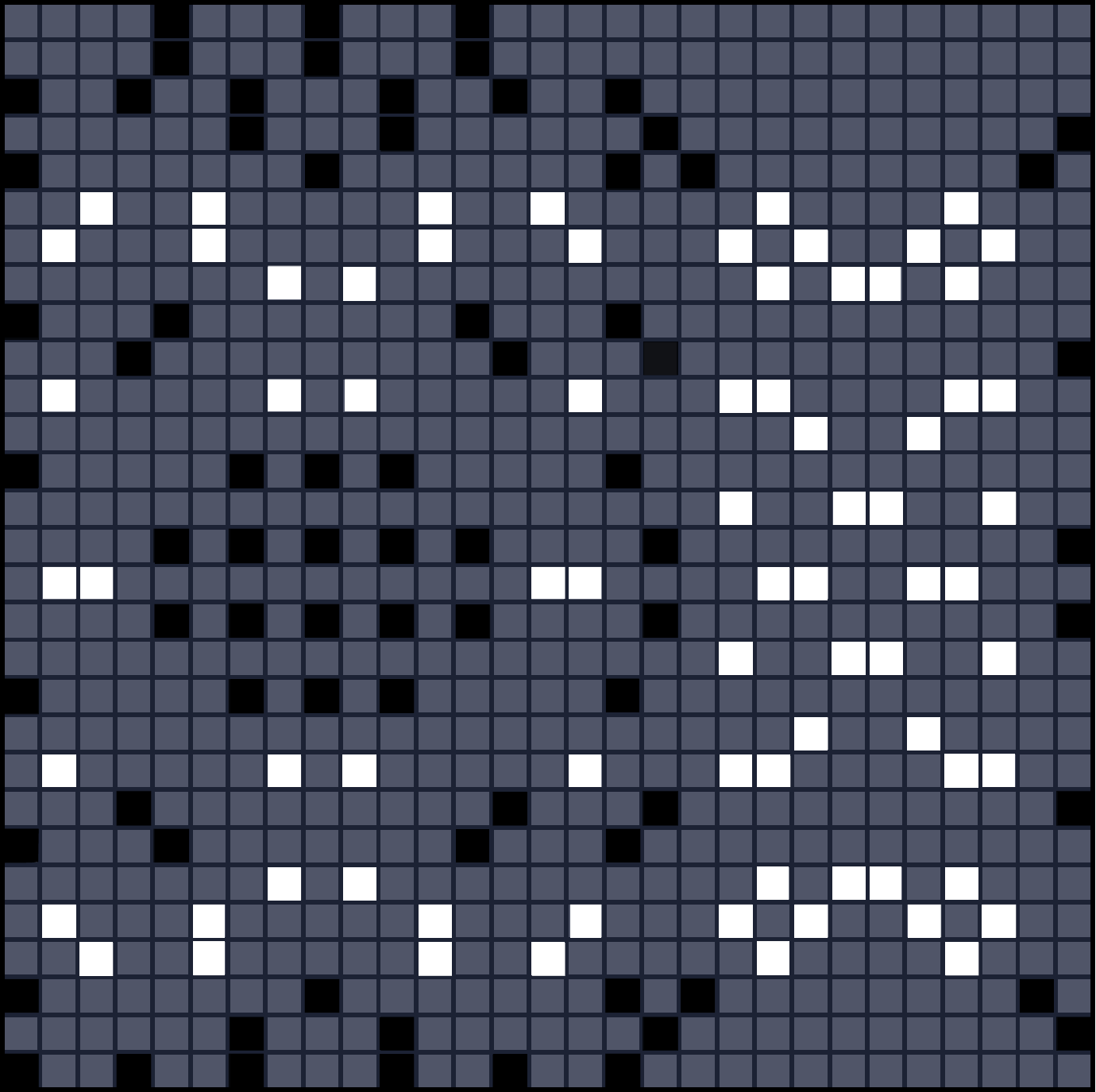}}
\end{figure}

\begin{figure}[H]
\captionsetup[subfloat]{labelformat=empty}
\subfloat[$n=31, \min\{|B|, |W|\}=88$]{\includegraphics[width = \textwidth]{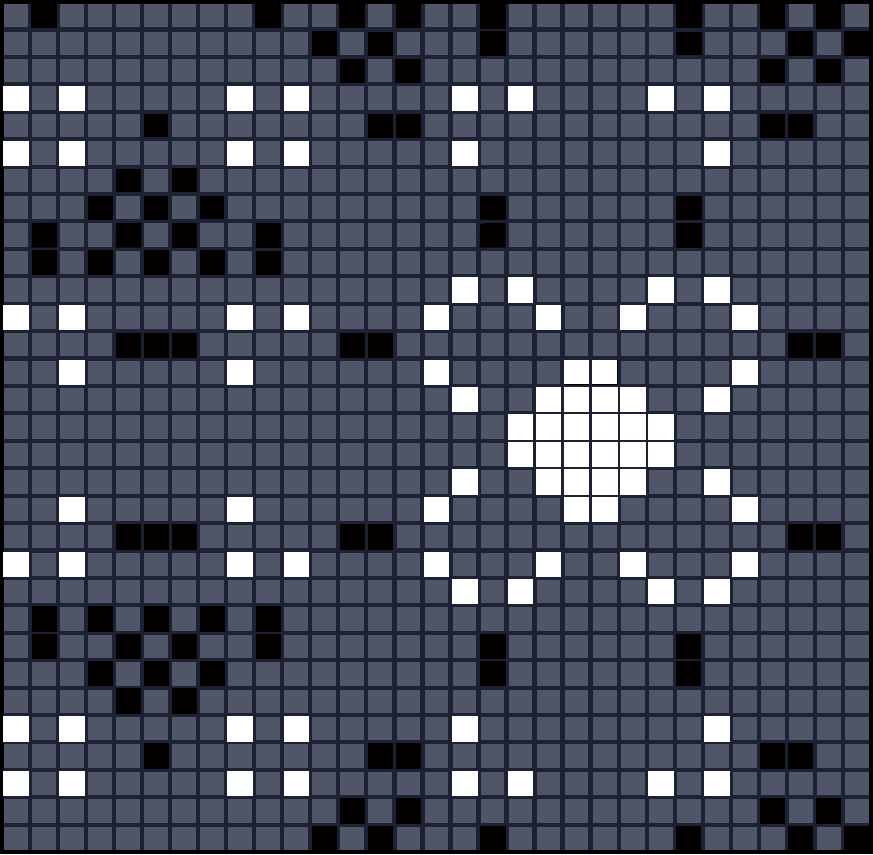}}
\end{figure}

\begin{figure}[H]
\captionsetup[subfloat]{labelformat=empty}
\subfloat[$n=33, \min\{|B|, |W|\}=101$]{\includegraphics[width = \textwidth]{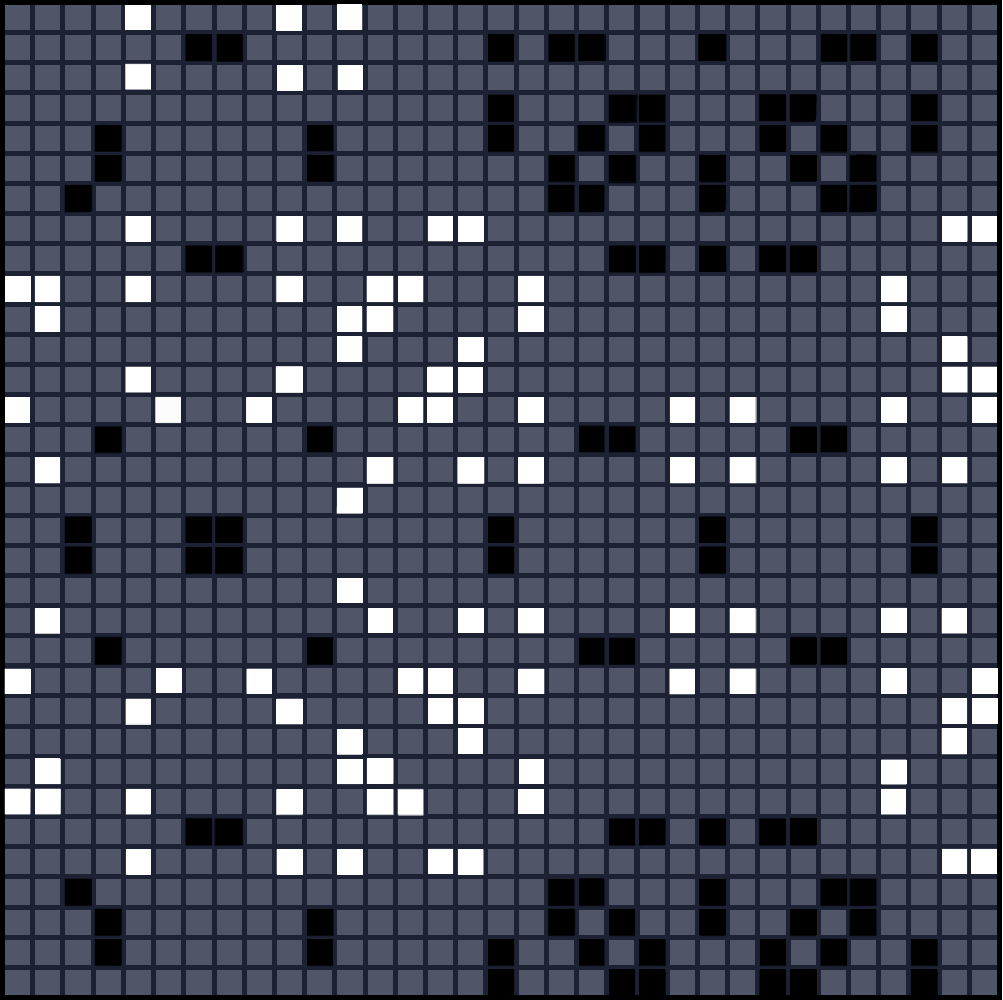}} 
\end{figure}

\begin{figure}[H]
\captionsetup[subfloat]{labelformat=empty}
\subfloat[$n=35, \min\{|B|, |W|\}=110$]{\includegraphics[width = \textwidth]{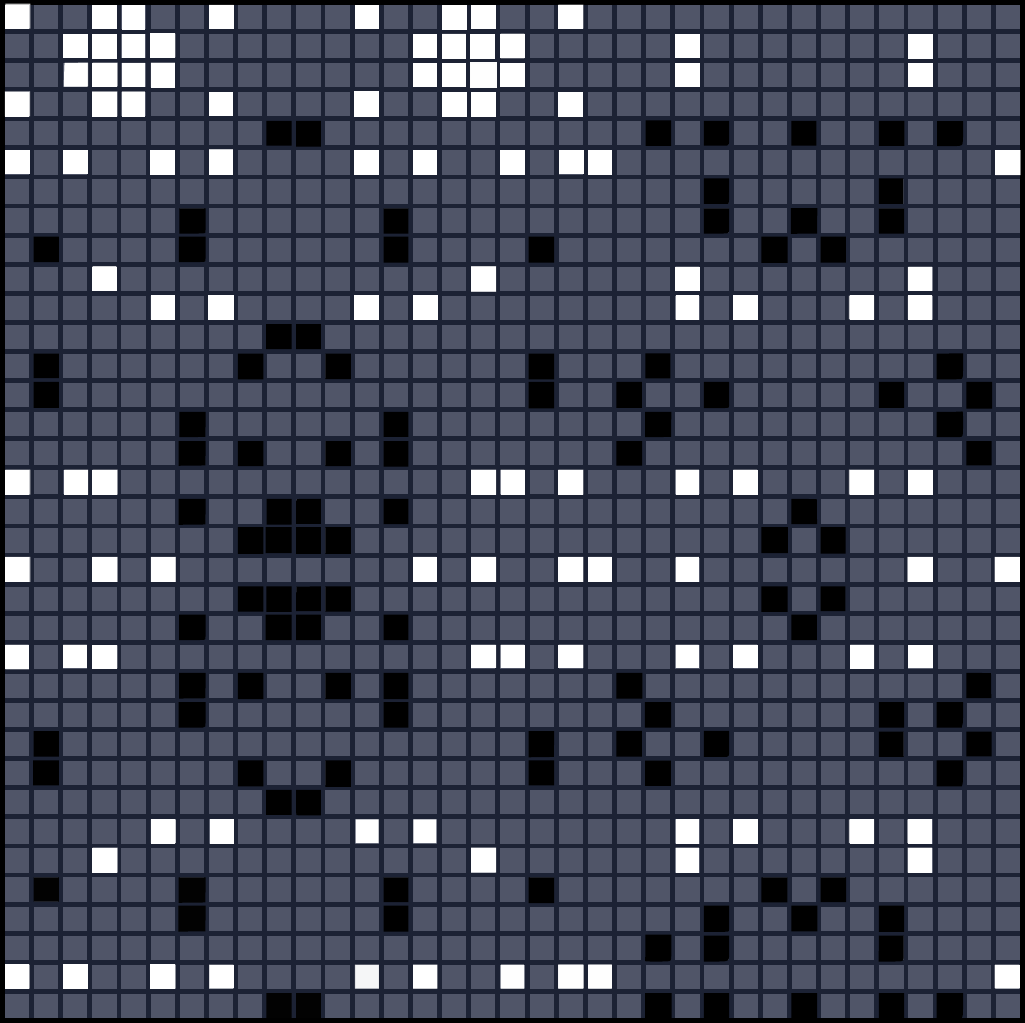}} 
\end{figure}

\begin{figure}[H]
\captionsetup[subfloat]{labelformat=empty}
    \subfloat[$n=37, \min\{|B|, |W|\}=126$]{\includegraphics[width = \textwidth]{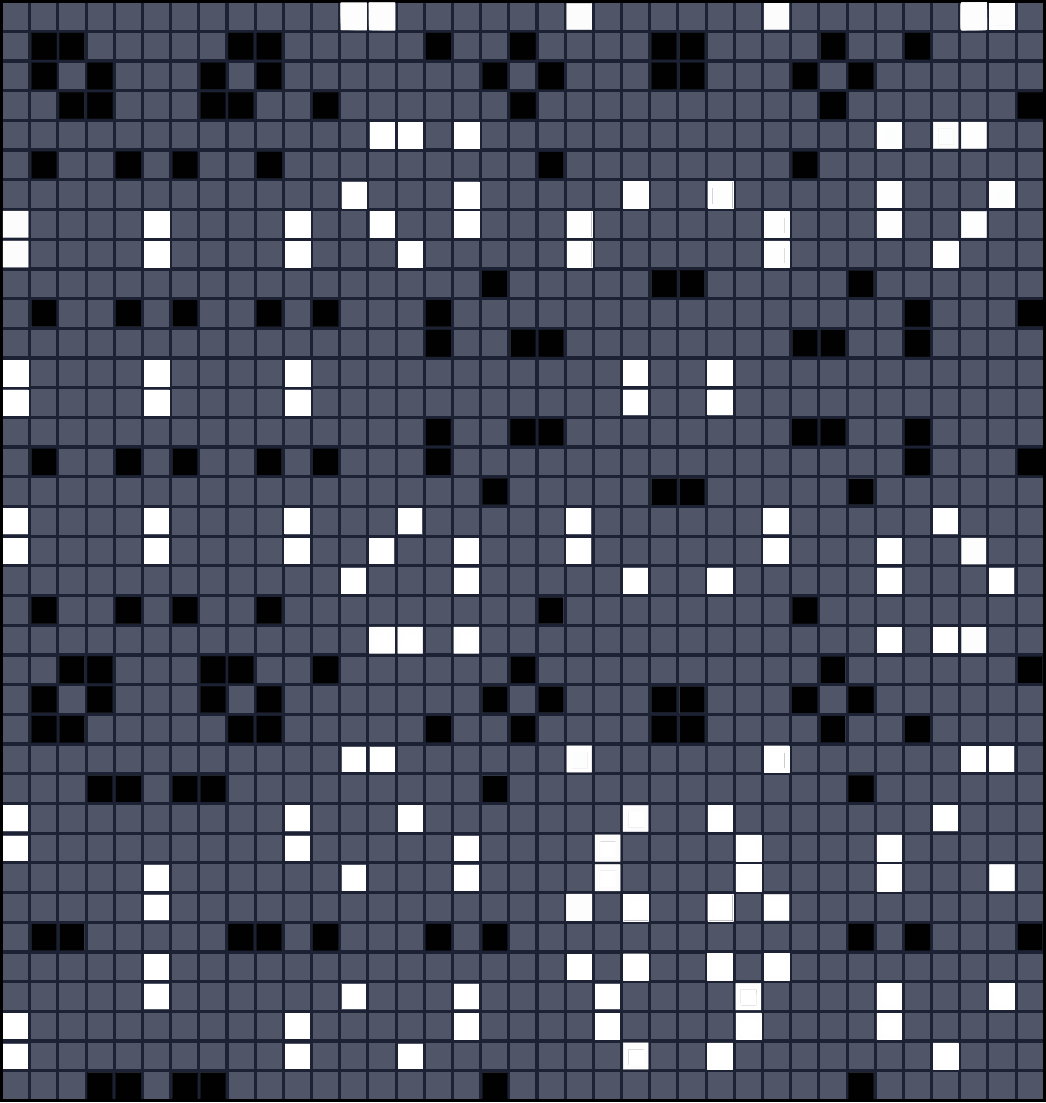}}
    \end{figure}

\begin{figure}[H]
\captionsetup[subfloat]{labelformat=empty}
\subfloat[$n=39, \min\{|B|, |W|\}=144$]{\includegraphics[width = \textwidth]{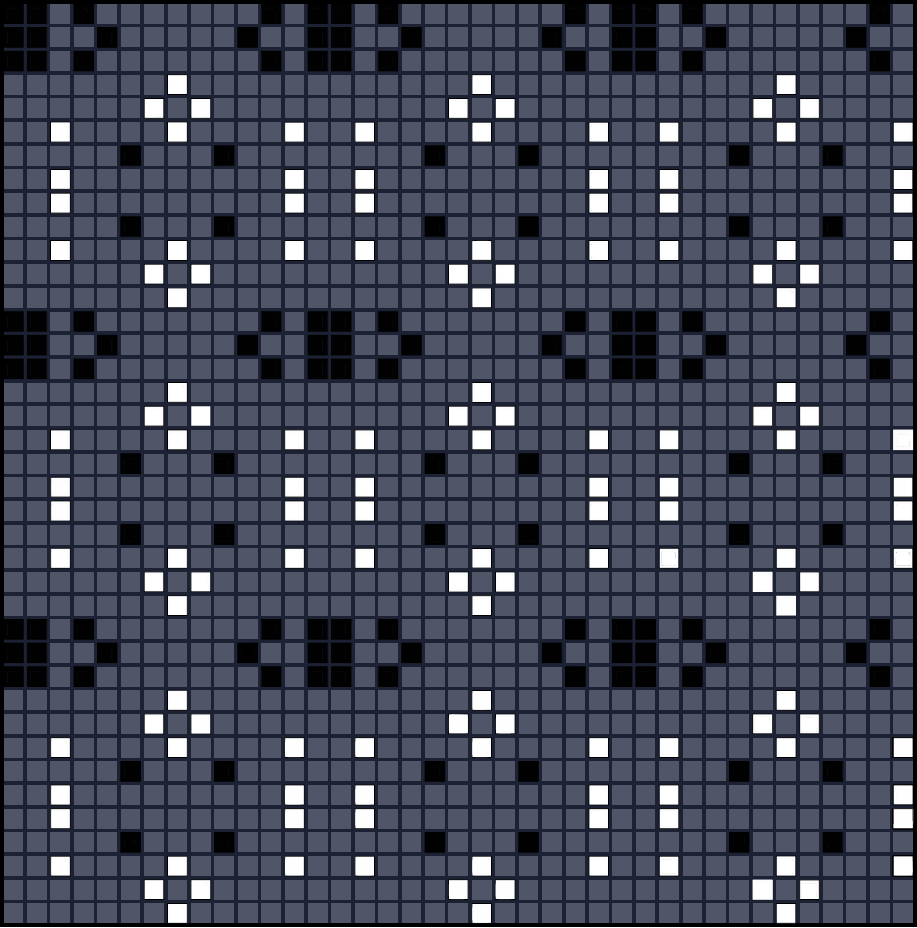}} 
\end{figure}

\begin{figure}[H]
\captionsetup[subfloat]{labelformat=empty}
\subfloat[$n=41, \min\{|B|, |W|\}=156$]{\includegraphics[width = \textwidth]{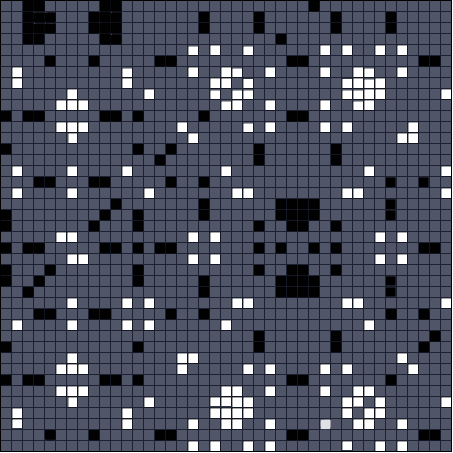}} \end{figure}

\begin{figure}[H]
\captionsetup[subfloat]{labelformat=empty}
\subfloat[$n=43, \min\{|B|, |W|\}=162$]{\includegraphics[width = \textwidth]{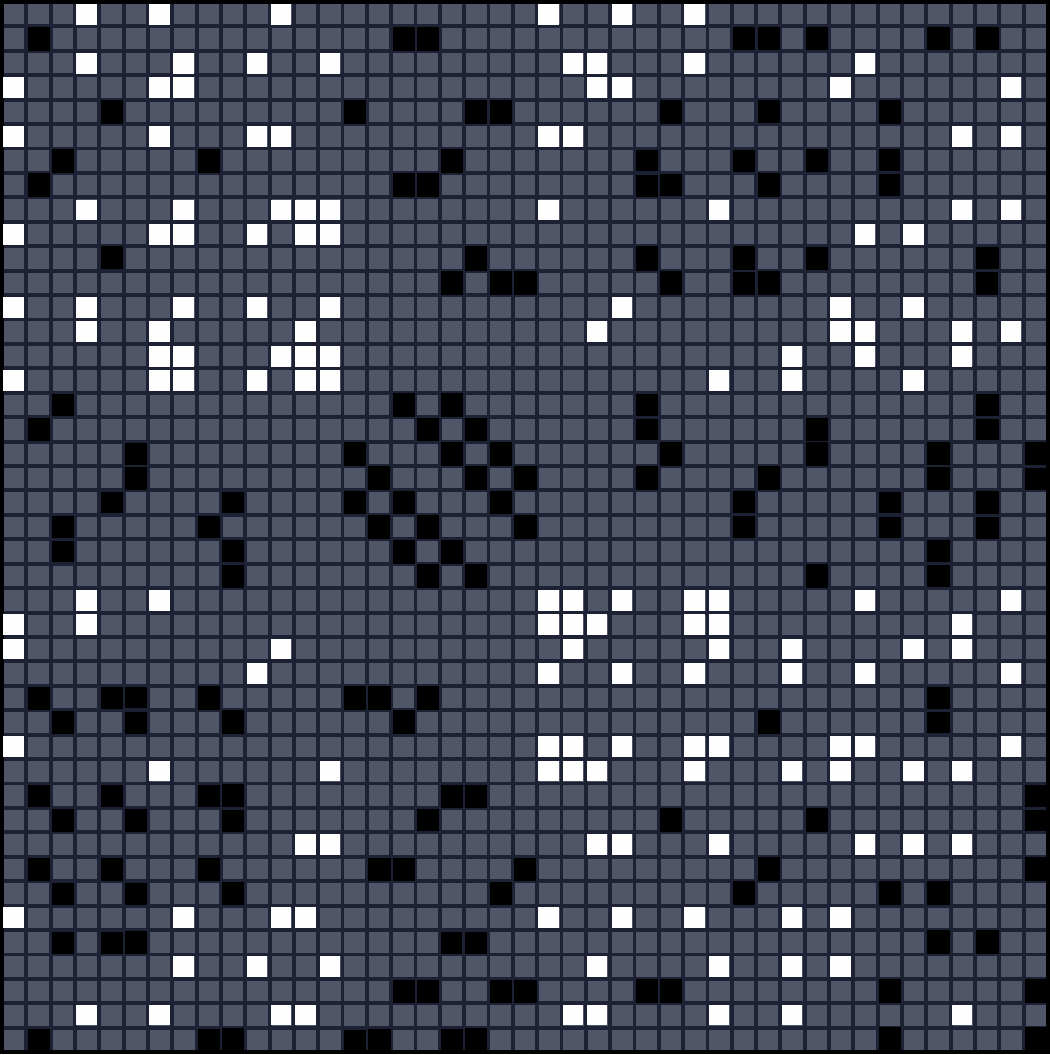}}  \end{figure}

\begin{figure}[H]
\captionsetup[subfloat]{labelformat=empty}
    \subfloat[$n=45, \min\{|B|, |W|\}=184$]{\includegraphics[width = \textwidth]{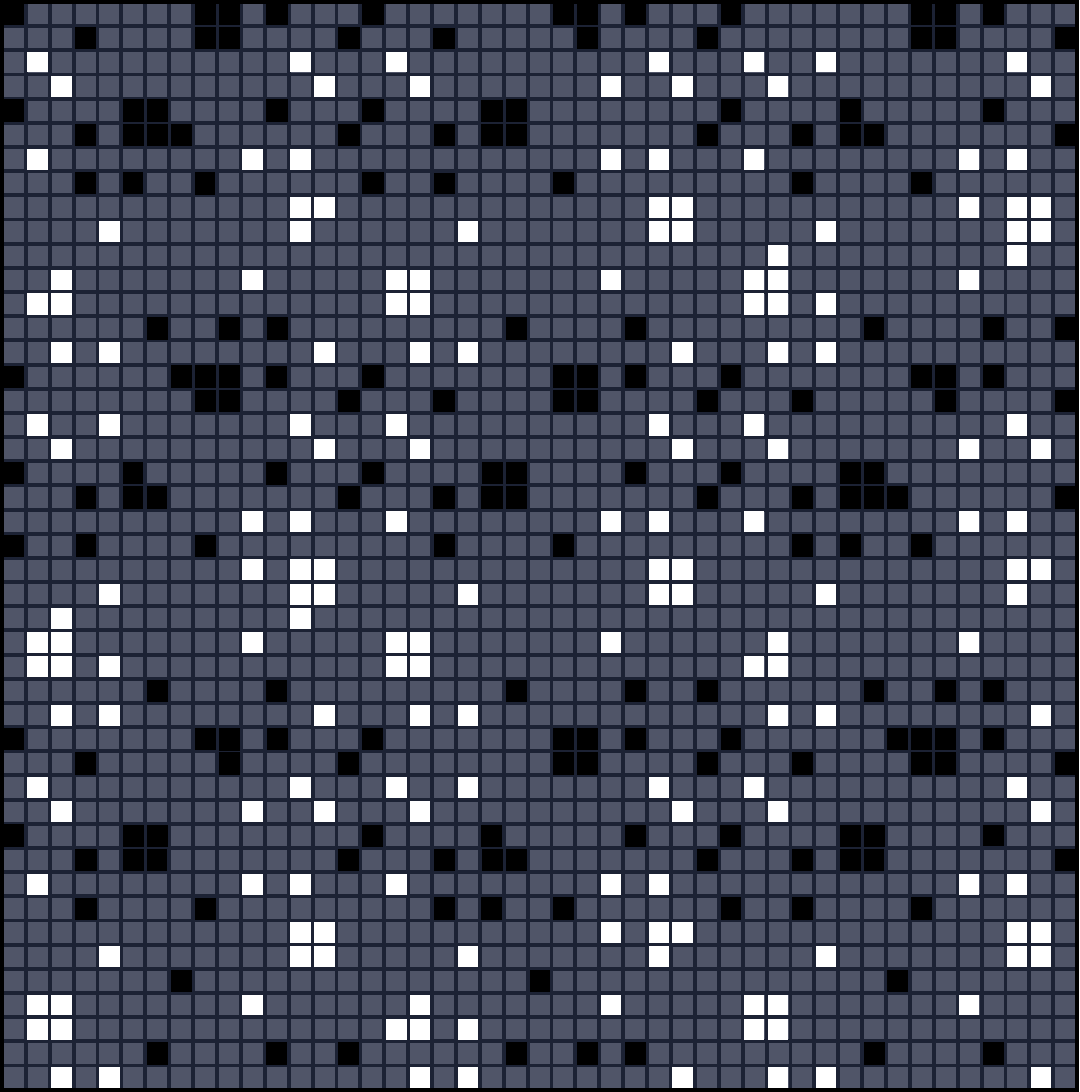}} 
\end{figure}

\begin{figure}[H]
\captionsetup[subfloat]{labelformat=empty}
\subfloat[$n=47, \min\{|B|, |W|\}=197$]{\includegraphics[width = \textwidth]{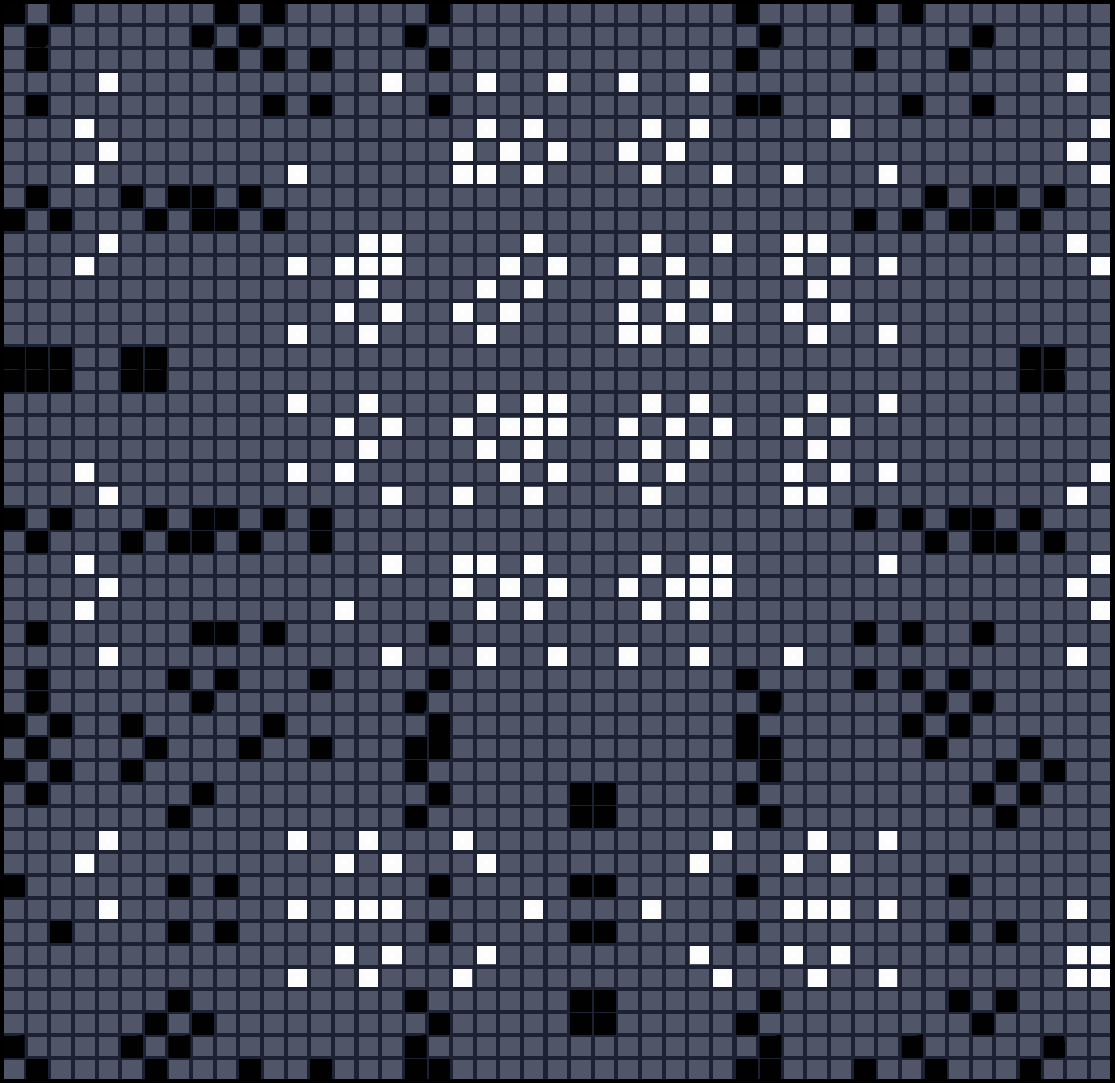}} 
\end{figure}

\begin{figure}[H]
\captionsetup[subfloat]{labelformat=empty}
\subfloat[$n=49, \min\{|B|, |W|\}=212$]{\includegraphics[width = \textwidth]{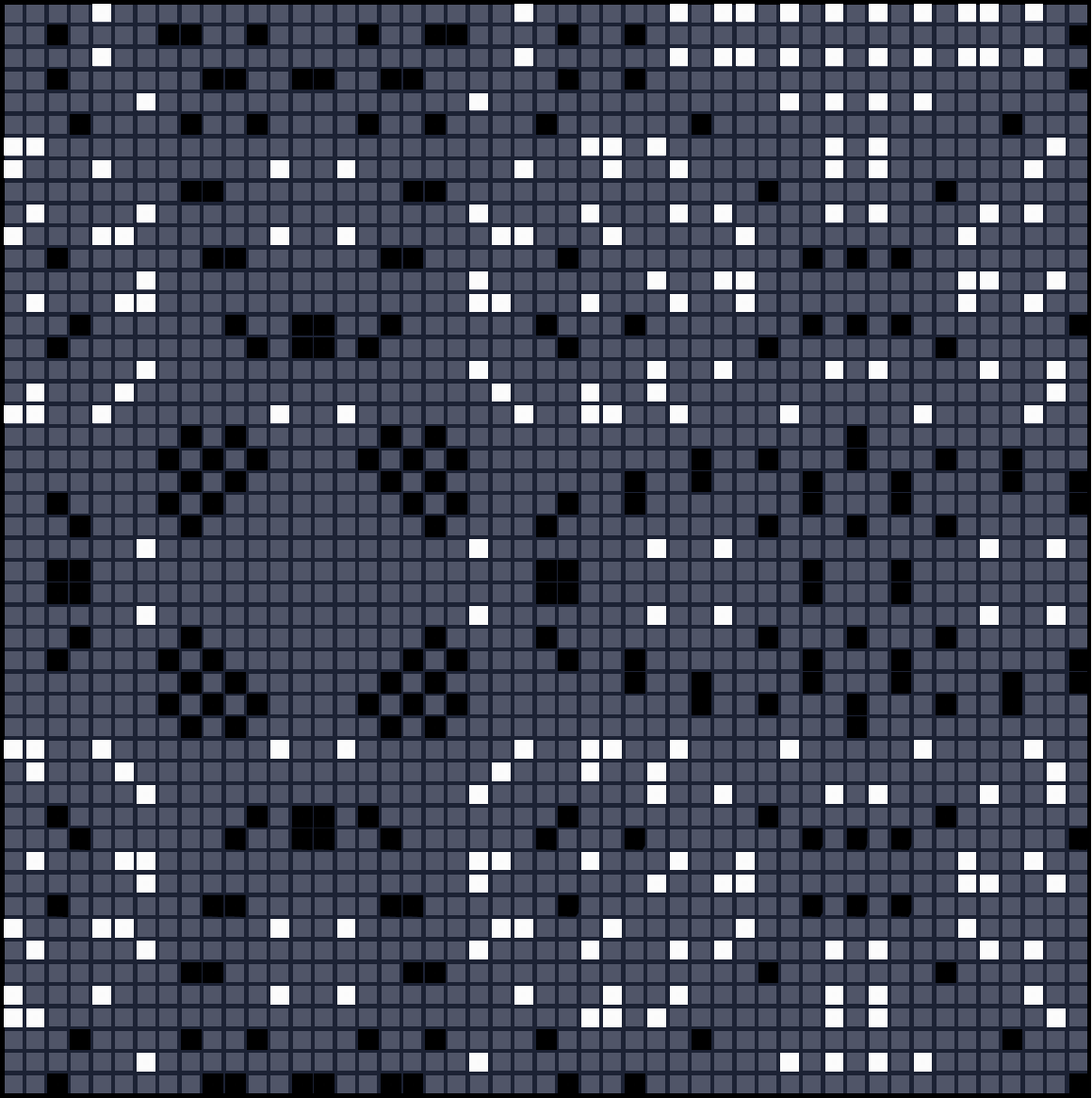}} \end{figure}

\begin{figure}[H]
\captionsetup[subfloat]{labelformat=empty}
\subfloat[$n=51, \min\{|B|, |W|\}=252$]{\includegraphics[width = \textwidth]{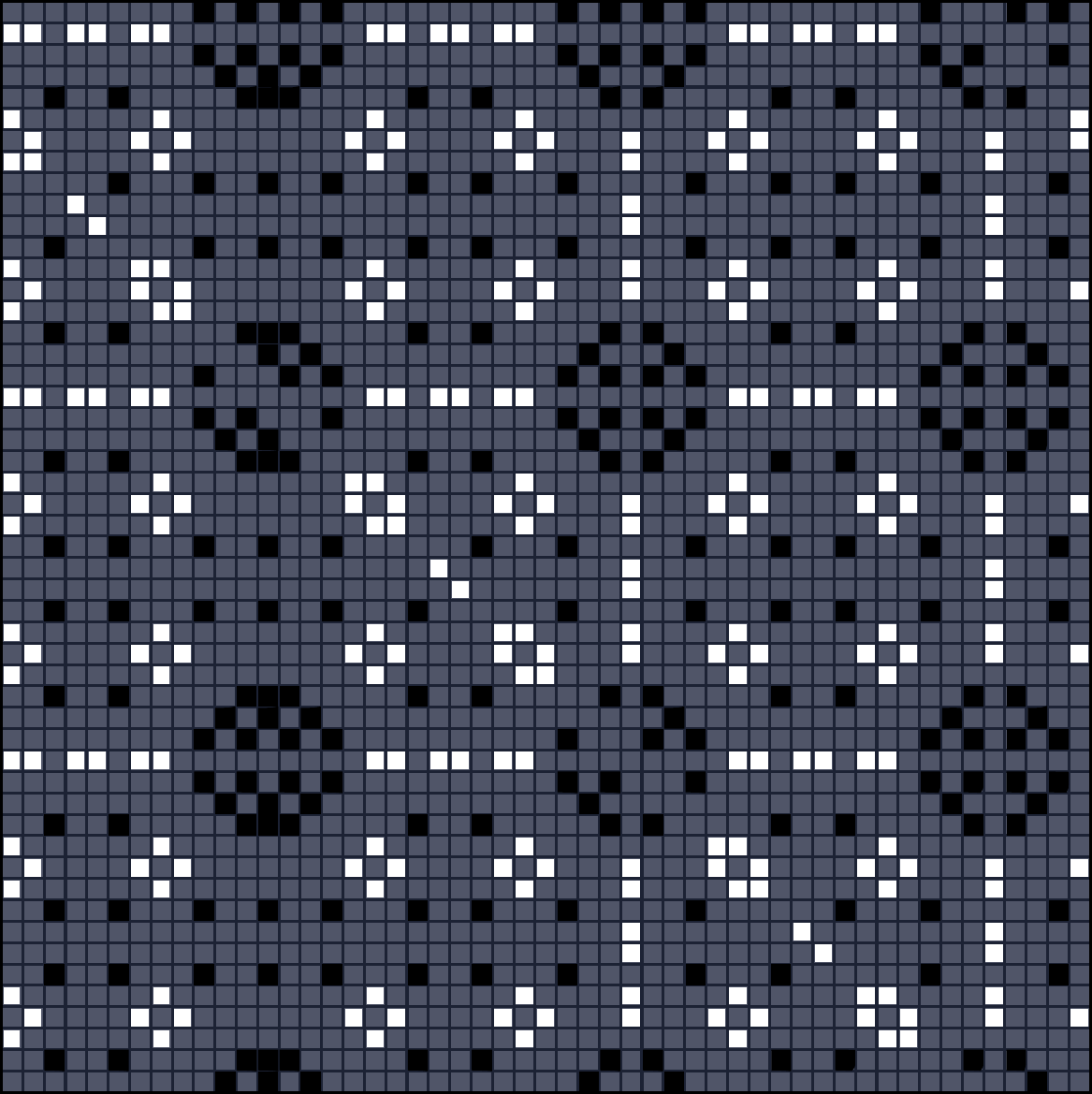}}  \end{figure}

\begin{figure}[H]
\captionsetup[subfloat]{labelformat=empty}
    \subfloat[$n=53, \min\{|B|, |W|\}=250$]{\includegraphics[width = \textwidth]{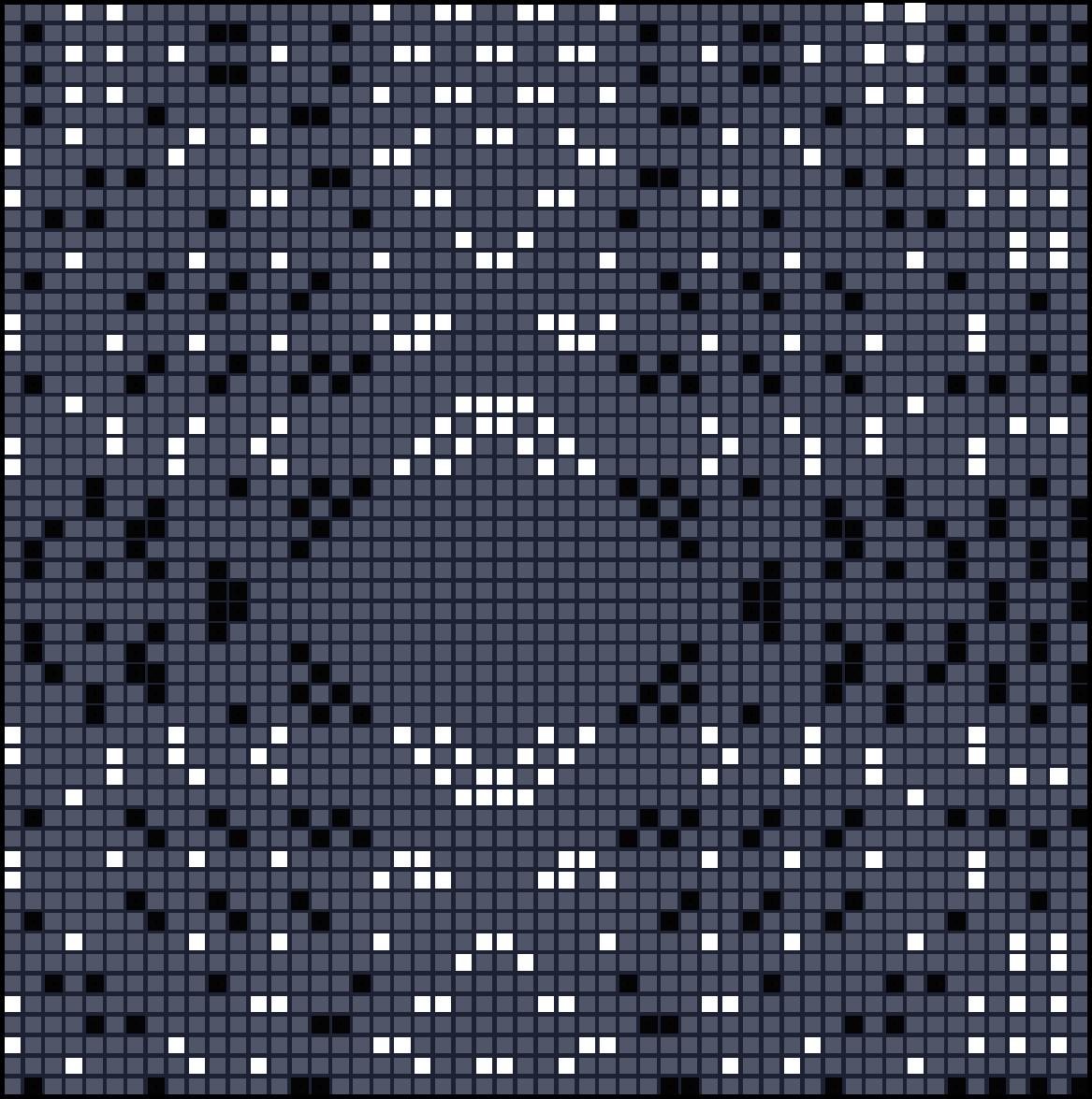}} 
    \end{figure}

\begin{figure}[H]
\captionsetup[subfloat]{labelformat=empty}
\subfloat[$n=55, \min\{|B|, |W|\}=264$]{\includegraphics[width = \textwidth]{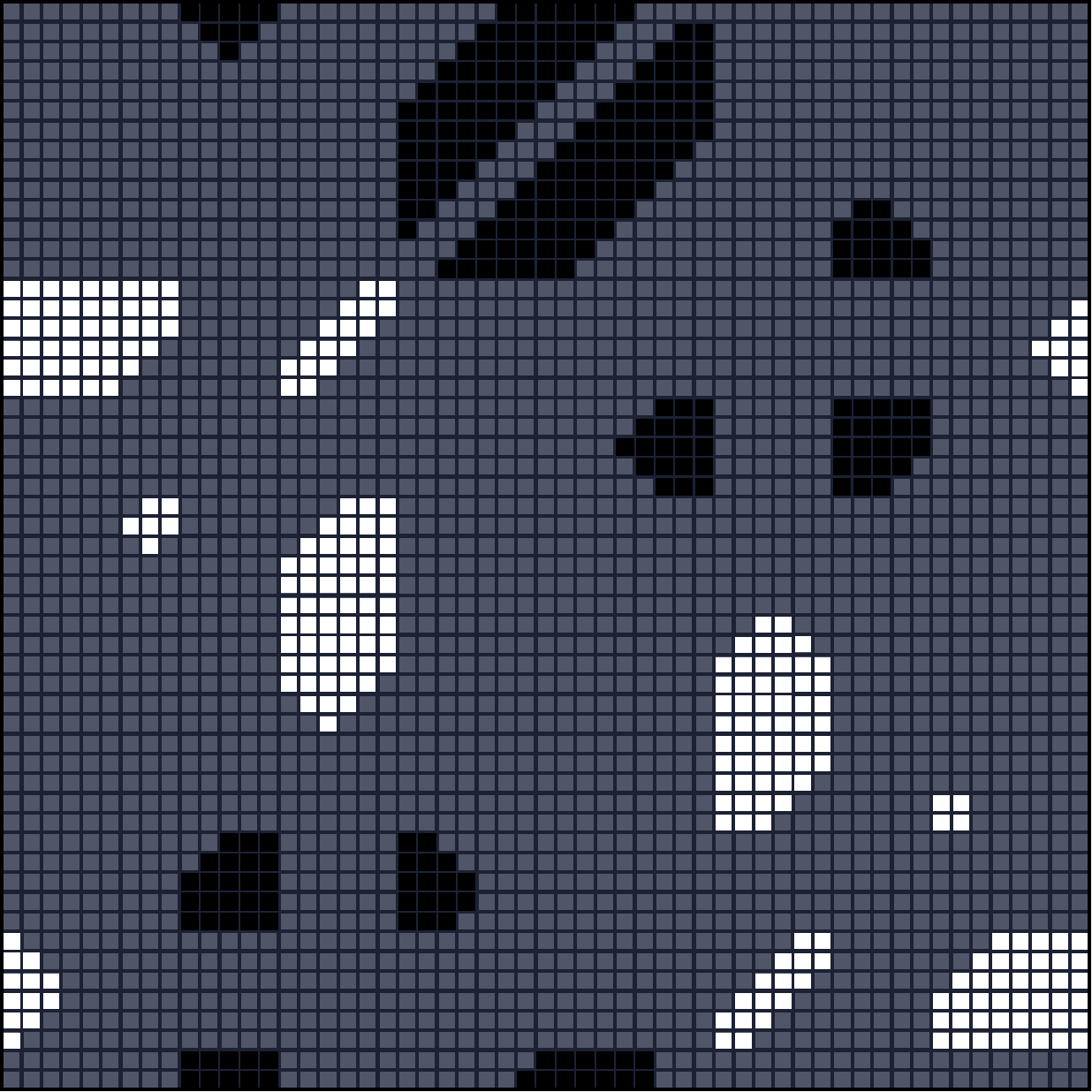}}
\end{figure}

\begin{figure}[H]
\captionsetup[subfloat]{labelformat=empty}
\subfloat[$n=57, \min\{|B|, |W|\}=285$]{\includegraphics[width = \textwidth]{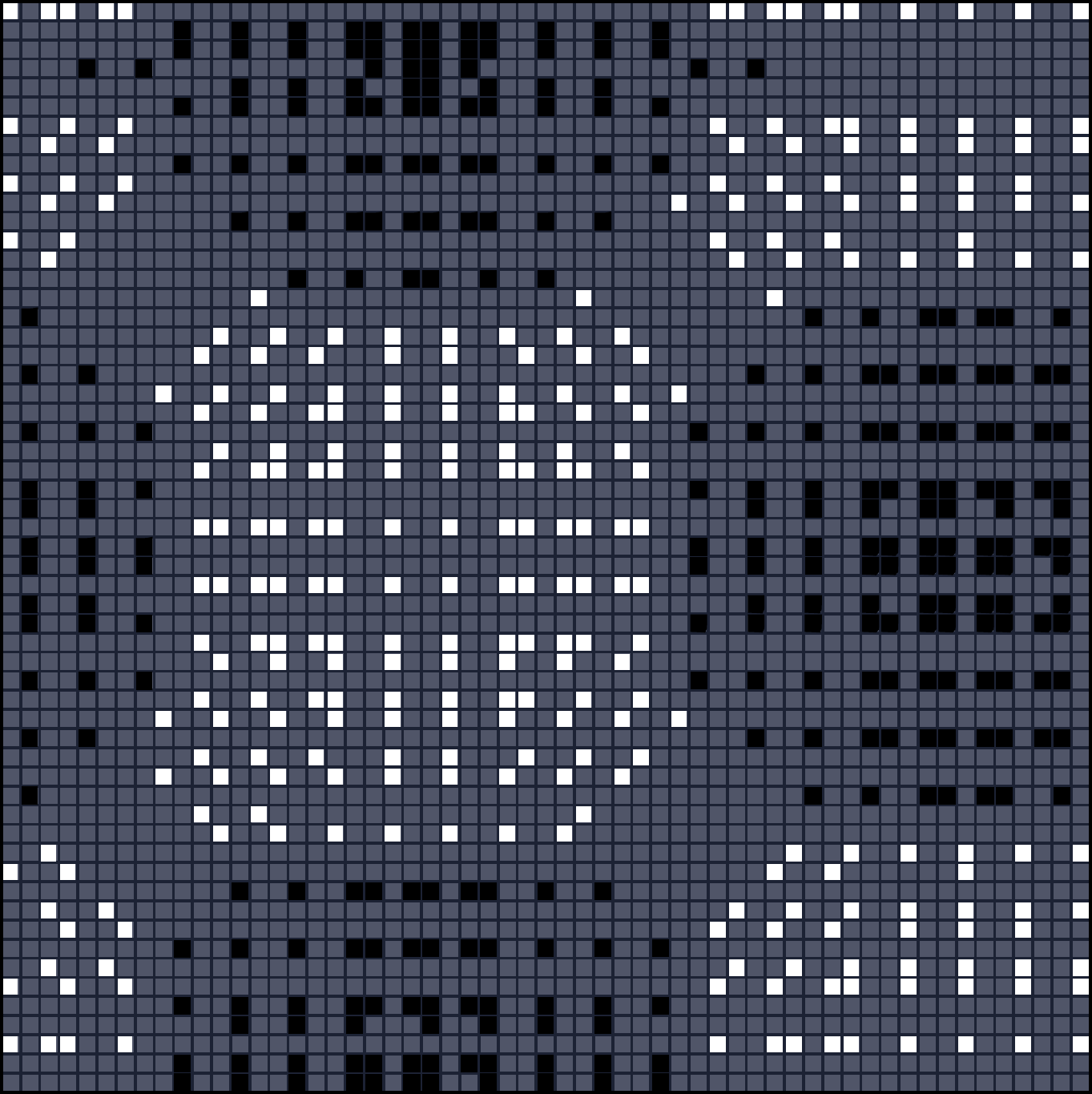}} \end{figure}

\begin{figure}[H]
\captionsetup[subfloat]{labelformat=empty}
\subfloat[$n=59, \min\{|B|, |W|\}=304$]{\includegraphics[width = \textwidth]{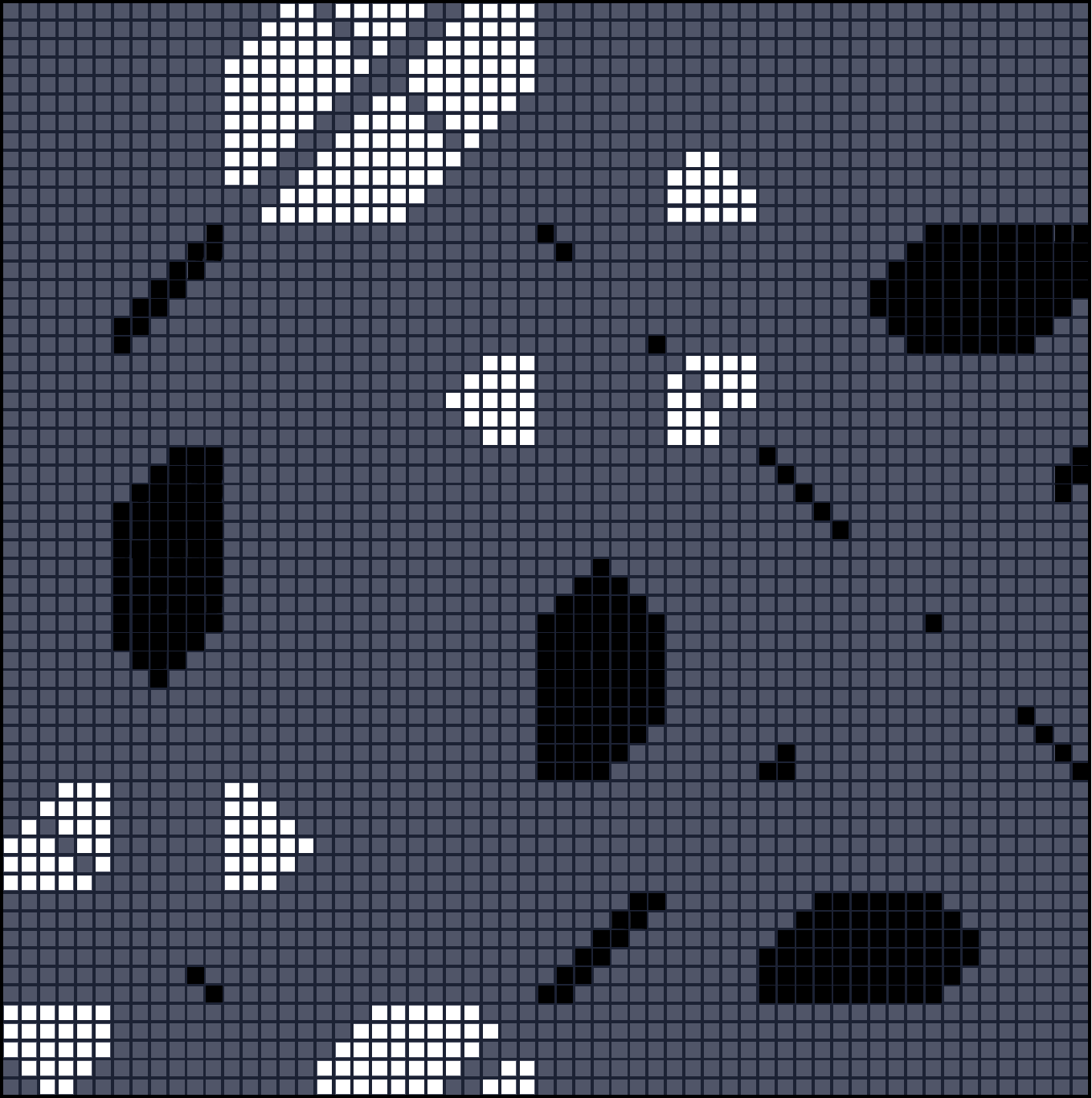}} 
\end{figure}

\begin{figure}[H]
\captionsetup[subfloat]{labelformat=empty}
\subfloat[$n=61, \min\{|B|, |W|\}=324$]{\includegraphics[width = \textwidth]{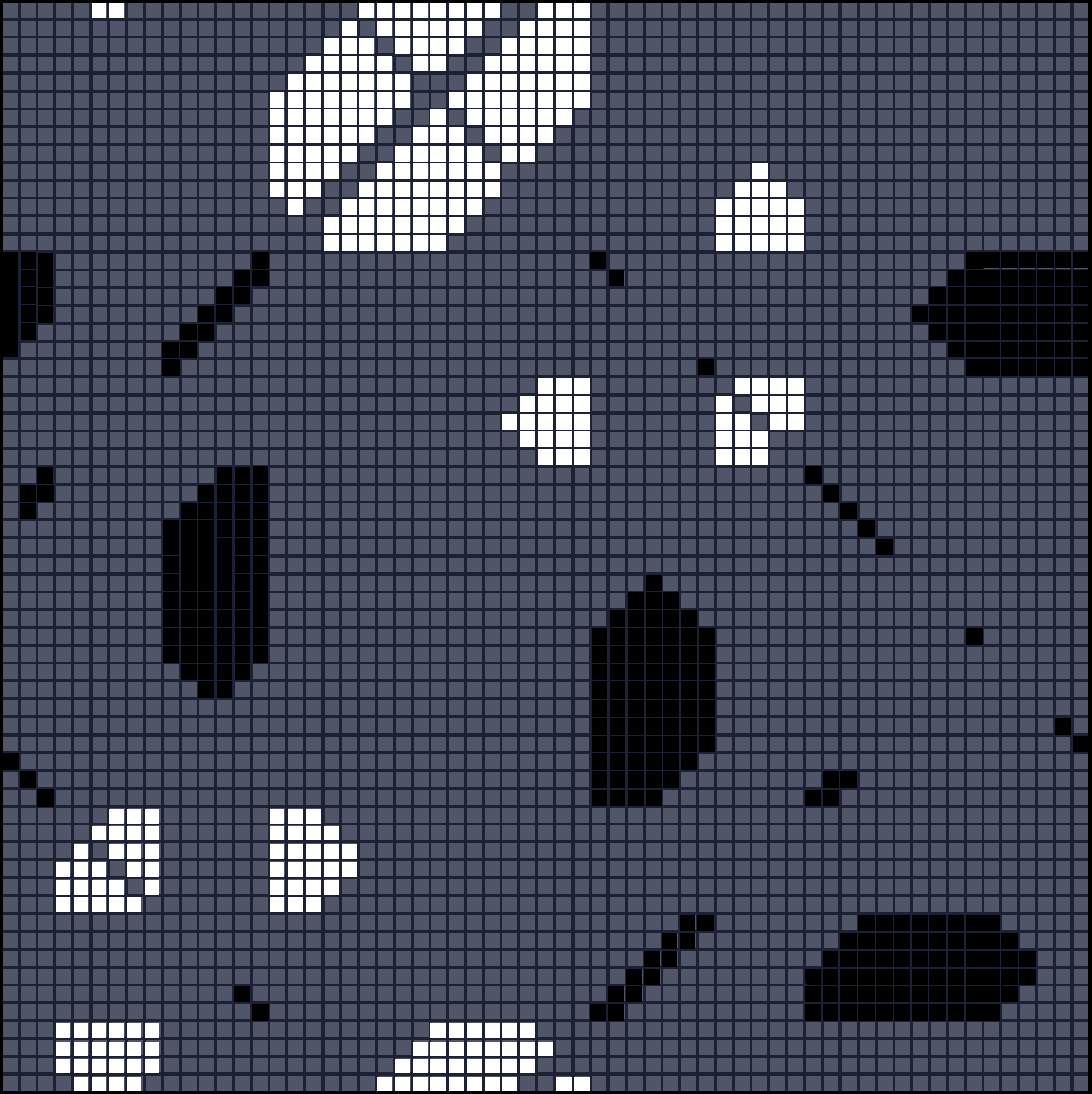}}
\end{figure}

\begin{figure}[H]
\captionsetup[subfloat]{labelformat=empty}
\subfloat[$n=63, \min\{|B|, |W|\}=348$]{\includegraphics[width = \textwidth]{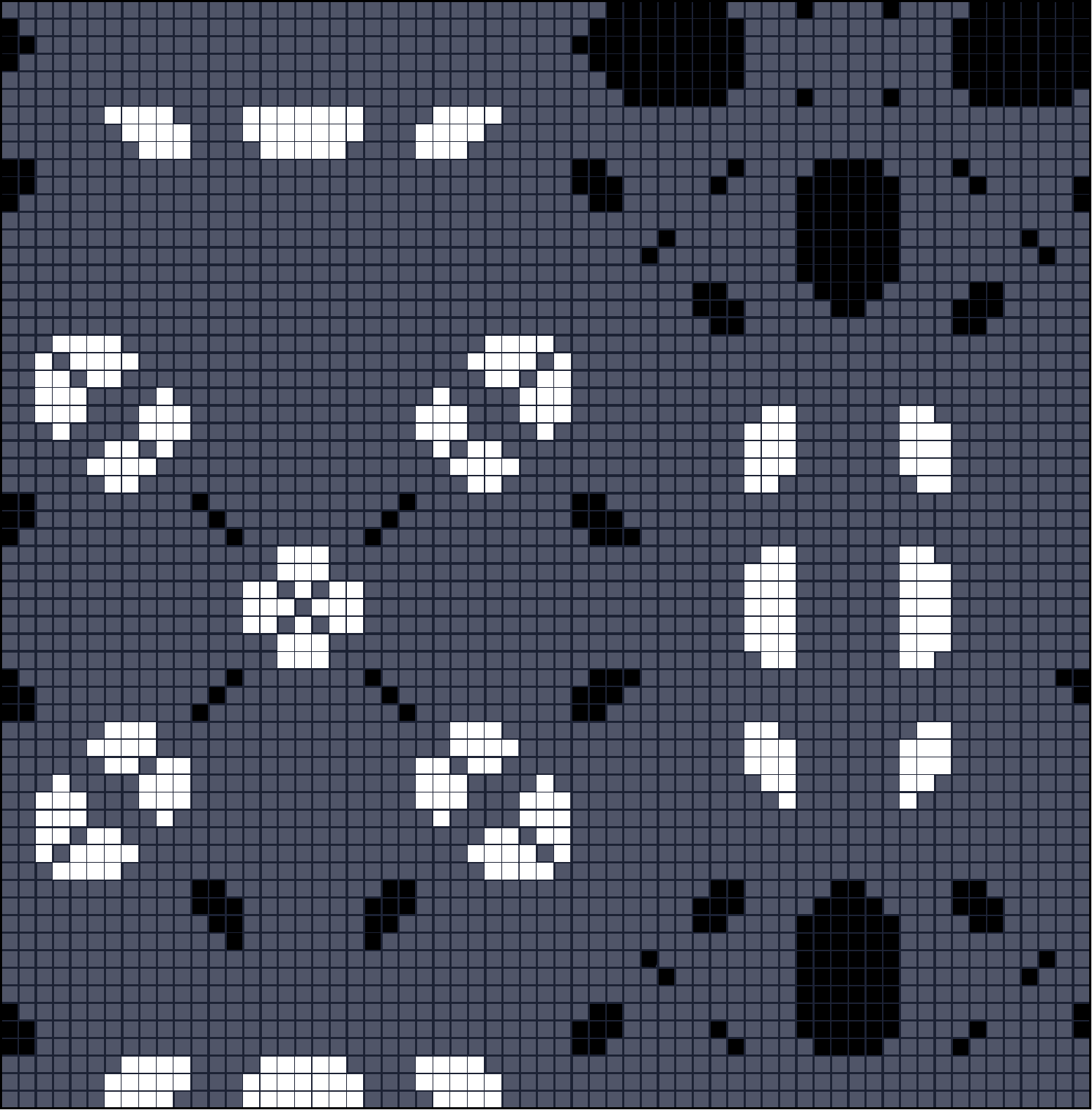}}
\end{figure}

\begin{figure}[H]
\captionsetup[subfloat]{labelformat=empty}
\subfloat[$n=63, \min\{|B|, |W|\}=348$]{\includegraphics[width = \textwidth]{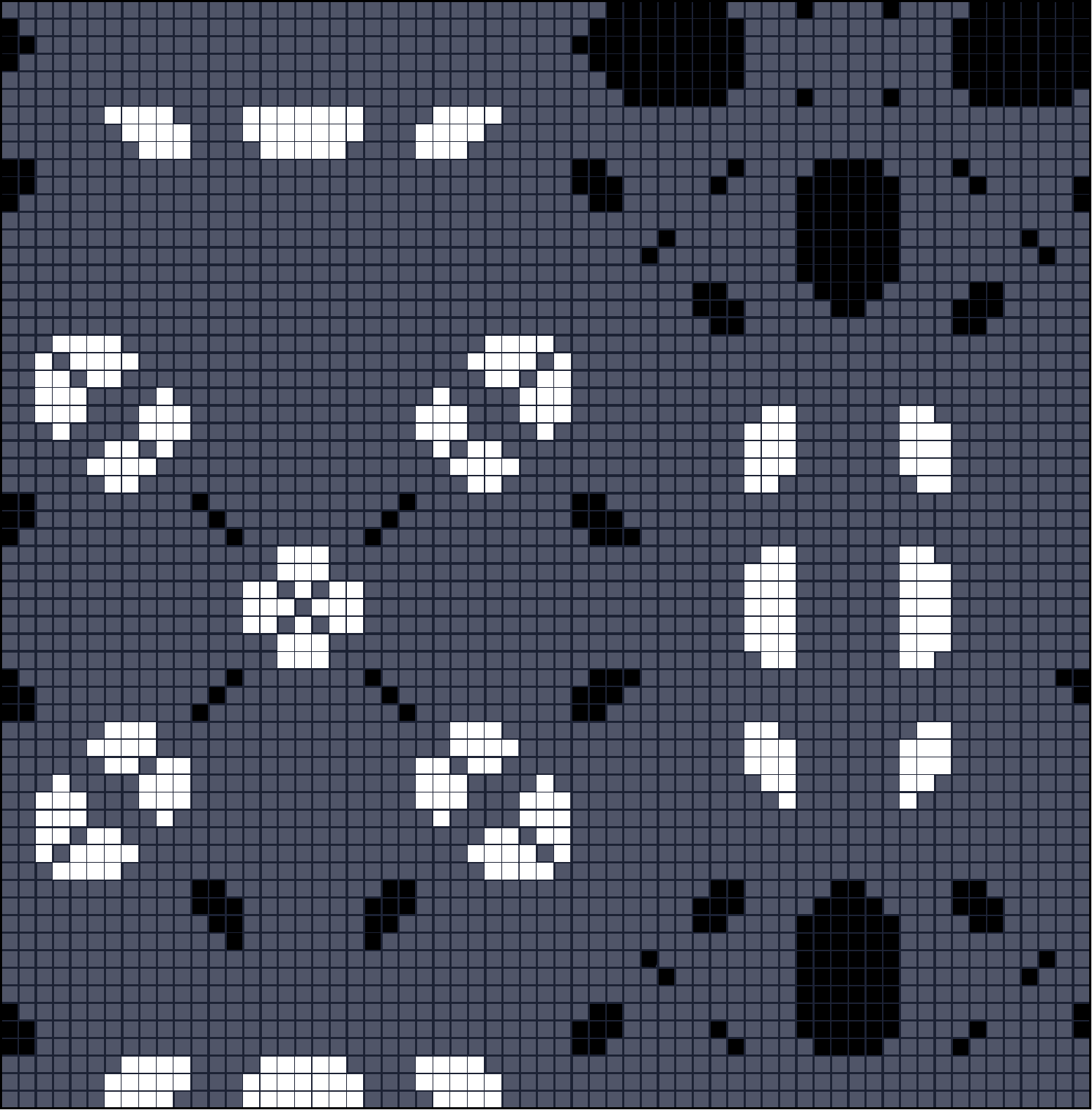}}
\end{figure}

\end{document}